\documentclass[12pt]{article}

\usepackage[T1]{fontenc}
\usepackage[utf8]{inputenc}
\usepackage{authblk}
\usepackage[margin=1in]{geometry}
\usepackage{amsmath,amssymb,amsthm}
\usepackage[usenames,dvipsnames,svgnames,table]{xcolor}
\usepackage{graphicx}
\graphicspath{{Figures/}}
\usepackage{mathrsfs}
\usepackage{enumitem}
\usepackage{bbm}
\usepackage{times}

\newtheorem{theorem}{Theorem}[section]
\newtheorem{lemma}[theorem]{Lemma}
\newtheorem{proposition}[theorem]{Proposition}
\newtheorem{definition}[theorem]{Definition}
\newtheorem{corollary}[theorem]{Corollary}
\newtheorem{remark}[theorem]{Remark}
\newtheorem{example}[theorem]{Example}

\newcommand{\bbR}{\mathbb{R}}
\newcommand{\bbX}{\mathbb{X}}
\newcommand{\bbZ}{\mathbb{Z}}
\newcommand{\bbN}{\mathbb{N}}

\newcommand{\bfB}{\mathbf{B}}
\newcommand{\bfE}{\mathbf{E}}
\newcommand{\bfI}{\mathbf{I}}
\newcommand{\bfM}{\mathbf{M}}
\newcommand{\bfU}{\mathbf{U}}
\newcommand{\bfe}{\mathbf{e}}
\newcommand{\clL}{\mathcal{L}}
\newcommand{\clQ}{\mathcal{Q}}
\newcommand{\clR}{\mathcal{R}}
\newcommand{\clH}{\mathcal{H}}

\newcommand{\srF}{\mathscr{F}}
\newcommand{\srX}{\mathscr{X}}
\DeclareMathOperator{\reach}{reach}
\DeclareMathOperator{\diam}{diam}

\DeclareMathOperator{\dist}{dist}
\DeclareMathOperator{\Unp}{\mathrm{Unp}}
\DeclareMathOperator{\Tan}{\mathrm{Tan}}
\DeclareMathOperator{\Nor}{\mathrm{Nor}}
\DeclareMathOperator{\interior}{\mathrm{int}}
\DeclareMathOperator{\Per}{Per}

\newcommand{\dmn}{n} %dimension; to be changed to d

\newcommand{\norm}[1]{\left\lVert#1\right\rVert}

\DeclareMathAlphabet{\pazocal}{OMS}{zplm}{m}{n}

\newcommand{\myell}{\,\hbox{\vrule height 7pt depth 0pt
		\vrule height 0.4pt depth 0pt width 6pt}\,}
		
\DeclareMathAlphabet{\pazocal}{OMS}{zplm}{m}{n}

\newtheorem*{thm*}{Theorem}

\usepackage{mathtools}

\usepackage{hyperref}
\usepackage{parskip}
\hypersetup{
	colorlinks=true,
	linkcolor=blue,
	filecolor=magenta,      
	urlcolor=cyan,
}

\usepackage[normalem]{ulem}

\definecolor{darkgrn}{rgb}{0, 0.75, 0}

\definecolor{bkclr}{rgb}{0.75, 0, 0.5}

\begin{document}
\title{\huge GEOMETRY OF A SET AND ITS\\ RANDOM COVERS}
\author[1]{Enrique Alvarado\thanks{ealvarado@math.ucdavis.edu}}
\author[2]{Bala Krishnamoorthy\thanks{kbala@wsu.edu}}
\author[2]{Kevin R. Vixie \thanks{vixie@speakeasy.net}}
\affil[1]{Department of Mathematics, University of California at Davis}
\affil[2]{Department of Mathematics and Statistics, Washington State University}
\maketitle

\begin{abstract}
Let $E$ be a bounded open subset of $\bbR^\dmn$.
We study the following questions: For i.i.d.~samples $X_1, \dots, X_N$ drawn uniformly from $E$, what is the probability that $\cup_i \bfB(X_i, \delta)$, the union of $\delta$-balls centered at $X_i$, covers $E$?
And how does the probability depend on sample size $N$ and the radius of balls $\delta$?
We present geometric conditions of $E$ under which we derive lower bounds to this probability.
These lower bounds tend to $1$ as a function of $\exp{(-\delta^\dmn N)}$.

The basic tool that we use to derive the lower bounds is a \emph{good} partition of $E$, i.e., one whose partition elements have diameters that are uniformly bounded from above and have volumes that are uniformly bounded from below.
We show that if $E^c$, the complement of $E$, has positive reach then we can construct a good partition of $E$.
This partition is motivated by the Whitney decomposition of $E$. 
On the other hand, we identify a class of bounded open subsets of $\bbR^\dmn$ that do not satisfy this positive reach condition but do have good partitions.
 
In 2D when $E^c\subset \bbR^2$ does not have positive reach, we show that the \emph{mutliscale flat norm} can be used to approximate $E$ with a set that has a good partition under certain conditions.
In this case, we provide a lower bound on the probability that the union of the balls \emph{almost} covers $E$.

\medskip
\noindent {\bfseries keywords:} Random covers, Whitney-type decomposition of sets, flat norm.

\end{abstract}

\clearpage

\section{Introduction}
We investigate relationships between the geometry of bounded open subsets $E$ of $\bbR^\dmn$ and the probability 
\begin{align}\label{eq:probability}
P(\bfB(\bbX, \delta) \supset E)
\end{align}
that the $\delta$-neighborhood $\bfB(\bbX, \delta) := \{y \in \bbR^\dmn : \mathrm{dist}(y, \bbX) \leq \delta\}$ (for $\delta > 0$) of a random sample $\bbX = \{X_j\}_{j=1}^N$ of $E$ will cover $E$.
%After going over some background in Geometric Measure Theory in \S2, we have two sections dedicated to the problem above.
How does this probability increase as we increase $N$ or $\delta$?
Ideally, we seek results that hold without any restrictive assumptions on $E$ being ``nice'', e.g., it being convex.
We derive new lower bounds on $P(\bfB(\bbX, \delta) \supset E)$ under the assumption that $E^c$, the complement of $E$, has positive \emph{reach}.  %(\cref{sec:whitney}).
The reach of a subset $A$ of Euclidean space is a measure of geometric regularity of $A$. 
It has played an indispensable role in areas such as signal processing and machine learning~\cite{niyogi2008finding, verma2012distance}.
In our work, the importance of reach is evident in Lemma \ref{lem:opening}, %of \S\ref{sec:whitney},
where assuming that the reach of $E^c$ is positive allows us to construct a partition $\clR$ of $E$ whose elements are not too small in measure and not too large in diameter.

This Lemma is used in Theorem \ref{thm:reach} to provide a lower bound to the probability in Equation (\ref{eq:probability}), which depends on $\delta$, the cardinality of our partition, and the number of samples $N$.
In practice, knowledge of the cardinality of $\clR$ can be used to determine the number of samples needed to ensure a high probability that $\bfB(\bbX, \delta)$ covers $E$.

We now state the Lemma and Theorem for convenience of the reader (we define reach in Definition~\ref{def:reach}). 
{
\renewcommand{\thetheorem}{\ref{lem:opening}}
  \begin{lemma} %NoNumber}
Let $E$ be a bounded open subset of $\bbR^\dmn$.
	If
	\[\reach(E^c) > \rho > 0,\] 
	then for any $0 < \delta < \rho$, there exists a finite partition $\clR$ of $E$ such that   
	\begin{align*}
		|R| \geq \frac{\delta^\dmn}{\dmn^{\dmn/2}}\quad \mathrm{and} \quad \mathrm{diam}(R) \leq 3\delta\quad \mathrm{for\ all}\ R \in \clR.
	\end{align*}
  \end{lemma}%NoNumber}
}

{
\renewcommand{\thetheorem}{\ref{thm:reach}}
  \begin{theorem} %NoNumber}
If $E$ is a bounded open subset of $\bbR^\dmn$ with $\reach(E^c) > \rho > 0$, then for any $\delta \in (0, \rho)$ there exists a natural number $M = M(\delta, E)$ such that for any independently and identically distributed sampling $\bbX_N = (X_i)_{i = 1}^N$ of $E$, 
\begin{align*}
P(\bfB(\bbX_N, 3\delta) \supset E) \geq 1 - M\exp\left(\frac{-\delta^\dmn}{\dmn^{\dmn/2}|E|} N\right).
\end{align*}
  \end{theorem}%NoNumber}
}

\vspace*{-0.07in}
\noindent Note that $P(\bfB(\bbX_N, 3\delta) \supset E) \to 1$ as $N \to \infty$.
  This probability also $\to 1$ by the factor of $\exp{(-\delta^\dmn)}$ when $\delta$ increases.

Although Lemma~\ref{lem:opening} and Theorem~\ref{thm:reach} only apply to bounded open subsets whose complements has positive reach, we study certain classes of open sets that admit lower bounds for the covering probability~(\ref{eq:probability}) but do not necessarily have complements with positive reach (Section \ref{sec:beyondwhitney}).
The first class of open sets are used in Corollary~\ref{cor:UminusA}, where we consider open subsets $E = F\setminus A$ of $\bbR^\dmn$ where $|A|$, the $\dmn$-dimensional Lebesgue measure of $A$,
is small and $F^c$ has positive reach.
Intuitively, one should be able to use the partition elements of $F$ obtained in the proof of Lemma~\ref{lem:opening} to get the partition elements of $E$.
The only change to the lower bound of the covering probability~(\ref{eq:probability}) is due to the fact that any partition element from the construction in Lemma~\ref{lem:opening} will be missing $|A|$ amount of its volume. 

We obtain a class of open subsets in $\bbR^2$ that admit lower bounds to a probability closely related to the covering probability~(\ref{eq:probability}) using the \emph{multiscale flat norm} \cite{IbKrVi2013} from Geometric Measure Theory (Section \ref{sec:flatnormR2}).
Theorem~\ref{prop:flatnormsufficient} states that if $E \subset \bbR^2$ is a set of finite perimeter and if a minimizer $S_\lambda$ of the multiscale flat norm $\srF_\lambda(\partial E)$ satisfies certain properties, then $E$ has a high probability of \emph{almost} being covered by the neighborhood of the samples.
%Precisely the theorem is stated as follows.
{
\renewcommand{\thetheorem}{\ref{prop:flatnormsufficient}}
  \begin{theorem}
Let $E \subset \bbR^2$ be an open set of finite perimeter. 
If $\lambda > \Lambda_E$ and $\delta > 0$ are picked such that 
\[|S_\lambda| < \frac{\delta^2}{2}\quad \text{ and } \quad 0 < \delta < \frac{1}{5\lambda} \quad \text{ for } \quad S_\lambda \in \srF_\lambda(\partial E),\]
then for $A := E\cap (E- S_\lambda)$ and $\alpha_\delta := \delta^2/2|E|$ there exists a good $\alpha_\delta$-almost partition $\clR_{A}$ of $E$ where 
\begin{align*}
P\left(\frac{|\bfB(\bbX, 3\delta) \cap E|}{|E|} \geq 1 - \alpha_\delta\right) \geq 1 - M\exp\left(\frac{-(\delta^2 2^{-1} - |S_\lambda|)}{|A|} N\right)
\end{align*}
where $N$ is the number of samples in $\bbX$ and $M$ is the cardinality of $\clR_{A}$.
\end{theorem}
}

In the case that $E = F \setminus A$ where $F^c$ has positive reach, $|A|$ is small, and $A$ is supported away from the boundary of $E$ (similar to the $E$ in Corollary~\ref{cor:UminusA}), in Theorem~\ref{thm:fnreach} we prove that for a certain range of scales, the \emph{flat norm minimizer} $S_\lambda$ of $\srF_\lambda(\partial E)$ is equal to $A$, and in particular can be used to ``fill'' $E$ back in with $F = E + S_\lambda$. 
{
\renewcommand{\thetheorem}{\ref{thm:fnreach}}
  \begin{theorem} %NoNumber}
Let $E = \bfE^2\myell (U \setminus A)$ where $U$ is a bounded, open subset of $\bbR^2$ with $\mathrm{reach}(\partial U) > \rho > 0$, and let $W$ be an open, compactly supported subset of $U$.
If 
\begin{align*}
0 < 2/\lambda < \rho,
\end{align*}
then 
\begin{align*}
\exists \delta > 0\ \text{such that if } A \subset\subset W \text{ with } |A| < \delta,\ \text{ then } \partial (\bfE^2 \myell U) \in \mathrm{arg\,}\srF_\lambda(\partial E).
\end{align*}
  \end{theorem}%NoNumber}
}
In the context of the previous results, we note that $\reach(U^c) \geq \reach(\partial U)$, and that Theorem~\ref{thm:fnreach} requires the stronger assumption that $\reach(\partial U) > \rho > 0$. %, whereas the previous results did not. 
We also note that this result is not a corollary of Theorem~\ref{prop:flatnormsufficient}, 
but is a result that presents an important example of how the multiscale flat norm works on the first class of open sets studied in Section \ref{sec:beyondwhitney}.

\subsection{Related work}

Problems on coverage by balls centered at random points are well studied in stochastic geometry \cite{ChStKeMe2013}.
The related topic of topology of random \v{C}ech or geometric complexes is an active area of research \cite{BoKa2018,BoWe2017}.

Janson studied a version of the problem where the sample $\bbX$ is drawn from a larger set $F$ such that $E \subset F^o$, its interior, thus avoiding the need to deal with boundary issues \cite{Ja1986}.
Finding the expected Hausdorff distance between random samples of uniform distributions and their support is closely related to our problem, and has been studied in general metric measure spaces. 
Reznikov and Saff~\cite[Theorem 2.1]{reznikov2016covering} gave a result for a finite positive Borel measure $\mu$ supported on a metric space $(\srX, m)$ that also satisfies 
$\mu(\bfB(x, r)) \geq \Phi(r)$ for all $x \in \srX$ and all $r < r_0$
for some $r_0 > 0$, and for some non-negative, strictly increasing, continuous function $\Phi : (0, \infty) \to \bbR$ satisfying $\lim_{r \to 0}\Phi(r) = 0$.
Their result states that if $\mu$ satisfies the above criteria, there exists constants $c_1, c_2, c_3$ and $\alpha_0$ such that for any $\alpha > \alpha_0$,
\begin{align}\label{eq:probpaper}
P\left[\sup_{y \in E} \dist(y, \bbX_N) \geq c_1\,\Phi^{-1}\left( \frac{\alpha \log N}{N}\right)\right] &\leq c_2\,N^{1 - c_3\alpha}.
\end{align}
As stated, their theorem requires $\Phi(r)$ to be positive \emph{for all} $x \in \srX$. 

We describe how our two main theorems relate to this work. 
Our Lemma~\ref{lem:opening} applies to $\mu = \clL^\dmn\myell E$ for a set $E$ such that its complement $E^c$ has positive reach.
Indeed this implies that $E$ may not, also, contain isolated points. 
However, our proof is simple enough to easily see how removing any set of small $\clL^\dmn$-measure from $E$ will effect our analogous bound~(\ref{eq:problowerbound}) of Theorem~\ref{thm:reach}.
The effect of removing such a set of small measure is discussed after the proof of Theorem~\ref{thm:reach}, and is stated with inequality~(\ref{eq:problowerbound2}).
In Theorem~\ref{thm:fnreach}, we show how we may use the multi-scale flat norm on certain types of ``badly behaved'' sets $E\subset \bbR^2$ (the metric measure space $(E, ||\cdot||_2, \clL^2\myell E)$ may not admit a strictly decreasing $\Phi$ such that $\clL^2\myell E(\bfB(x, r)) \leq \Phi(r)$ for all $x \in E$)
to identify a nice enough set that can partition with the method of Lemma~\ref{lem:opening}. 

Our work is related to the following {\it random coverage problem}.
For a compact region $E$ in $\bbR^\dmn$, what is the probability that $E$ is fully covered by a union of balls of radius $r$ centered on $N$ points placed independently and uniformly on $E$ in the limit as $N \to \infty$ and as $r(N) \to 0$ in an appropriate manner?
Penrose~\cite{penrose2021random} gives a limiting behavior of the probability $P\left[R_{N, k} < r(N)\right]$
for any $k \in \bbN$ as one increases the number of sample points $\bbX_N$ and for any decreasing sequence $r(N) \to 0$ as $N\to \infty$, where $R_{N, k}$ denotes the smallest radius necessary for $\{\bfB(x, R_{N, k})\}_{x \in \bbX_N}$ to cover $E$ $k$-times.
These results hold when $E$ has {\it smooth boundary} for $n \geq 2$, or when $E$ is a polytope for $n \leq 3$. 
In our work, we give quantitative results on the probability~(\ref{eq:probability}) where, in contrast, we do not consider covering $E$ more than a single time and where $\delta$ is \emph{independent} of $N$.
Furthermore, our general dimensional results hold for bounded open subsets of $\bbR^\dmn$ with complements of \emph{positive reach}.
There are several classes of sets that have complements of positive reach but do not have smooth boundary,
e.g., when $E$ consists of two disks with intersecting interiors.
And unlike most previous work, we also study sets that may not satisfy the positive reach condition (see Figure \ref{fig:EminusA} for a set whose complement does not have positive reach), but could be \emph{almost} covered.

\section{Background}

Since we use only elementary results on probability, most of this section provides relevant background from geometric measure theory.
For the concepts that are closely related to \emph{reach}, we use the notation from Federer's seminar paper~\cite{federer-1959-1} on \emph{curvature measures}. 
For concepts related to \emph{currents}, we use the notation in Federer's comprehensive treatise on \emph{Geometric Measure Theory}~\cite{federer-1969-1} as well as Vixie's paper on properties of flat norm minimizers~\cite{Vi2007}.

\noindent We denote open balls of radius $r$$>$$0$ centered at $x \in \bbR^\dmn$ as $\bfU(x, r)$ and closed balls as $\bfB(x, r)$.

\setcounter{theorem}{0}

\begin{definition}
  Let $E$ be an open bounded subset of $\bbR^\dmn$ and suppose that $\bbX = (X_i)_{i = 1}^N$ are independently and identically distributed random variables with respect to the uniform distribution of $\Omega$. 
i.e., $X_i\sim \mathrm{Unif}(\Omega)$ for all $i = 1,\dots, N$.\\
\end{definition}

\subsection{Sets of positive reach}

Let us now go over some definitions and properties about sets with positive reach. 
While we may not use all of these properties in our proofs, we nevertheless state them since they give helpful intuition to the reader.

\begin{definition}\label{def:reach}
We define the {\bfseries distance function} $\delta_A :\bbR^\dmn\to \bbR$ of a subset $A$ of $\bbR^\dmn$ as
\[\delta_A(x) = \mathrm{dist}(x, A) := \inf_{a \in A} |x - a| \, \text{ for } x \in \bbR^\dmn.\]
In addition, we define 
\[\Unp(A)\]
to be the set of all points $x \in \bbR^\dmn$ such that there exists a unique point of $A$ nearest to $x$, and the {\bfseries nearest point map} $\xi_A: \Unp(A) \to A$ as the one
mapping $x\in \Unp(A)$ to the unique $a \in A$ such that $\delta_A(x) = |x - a|$. 

If $a \in A$, then 
\[\mathrm{reach}(A, a) = \sup\{r : \bfU(a, r) \subset \Unp(A)\}.\]
In addition, we define the {\bfseries reach} of $A$ to be 
\[\reach(A) = \inf\{\reach(A, a) : a \in A\}.\]
\end{definition}

\begin{definition} \label{def:outinormal}
Let $\gamma: \mathbb{S}^1 \to \mathbb{R}^2$ to be a $C^2$ isometric embedding of the circle, and let $\Gamma$ denote its image. 
For any $\epsilon > 0$, we define the {\bfseries outer-normal map}, $\alpha_\epsilon: \mathbb{S}^1\to\mathbb{R}^2$ as $\mathbf{v}_\epsilon(s) = \gamma(s) + \epsilon N(s)$, where $N(s)$ denotes the outer unit normal of $\Gamma$ at $\gamma(s)$.
Similarly, we may define the {\bfseries inner-normal map} $\bar{\mathbf{v}}_\epsilon$ with the inner unit normal of $\Gamma$ instead. 
We will denote the image of the outer-normal map and inner-normal map as $\Gamma_\epsilon$, and $\bar{\Gamma}_\epsilon$ respectively. 

We define the {\bfseries injectivity radius} of $\gamma$ to be the supremal $\epsilon$ for which both $\mathbf{v}_\epsilon$ and $\bar{\mathbf{v}}_\epsilon$ are injective on $\mathbb{S}^1$.
\end{definition}

\begin{remark}
If $\gamma : \mathbb{S}^1 \to \bbR^2$ is an $C^2$ isometric embedding, then $\reach(\gamma(\mathbb{S}^1))$ equals the injectivity radius of $\gamma$.
\end{remark}

\begin{definition}
Let $A \subset \bbR^\dmn$ and $a \in A$.  
We define the {\bfseries tangent cone}
\[\Tan(A, a)\]
{\bfseries  of $A$ at $a \in A$}
to consist of all {\bfseries tangent vectors} $u \in \bbR^\dmn$ where either $u = 0$, or for every $\epsilon > 0$ there exists a point $b \in A$ such that 
\[0 < |b - a| < \epsilon \quad \text{ and } \quad \left| \frac{b - a}{|b - a|} - \frac{u}{|u|}\right| < \epsilon.\]

\noindent We will also consider the {\bfseries normal cone of $A$ at $a \in A$} defined as
\[\Nor(A, a) = \{v \in \bbR^\dmn : v \cdot u \leq 0,\text{ for any } u \in \Tan(A, a)\},\]
whose elements are called {\bfseries normal vectors of $A$ at $a\in A$}.
\end{definition}

The following fact was stated by Federer~\cite[\text{Remark 4.20}]{federer-1959-1}, and was proven by Lytchak~\cite[\text{Proposition 1.4}]{lytchak2005almost}.
If $A \subset \bbR^\dmn$ is a topological manifold of dimension $0 < k < n$ and $\reach(A) > 0$ then $A$ is a $k$-dimensional $C^{1,1}$ submanifold of $\bbR^\dmn$. 
That is, it can be locally represented as the graph of a $C^1$ function whose differential is Lipschitz. 

We will use the following property of reach to prove Theorem~\ref{thm:fnreach}.
Its proof may be found in Federer's paper on \emph{Curvature Measures}~\cite{federer-1959-1}, along with the proofs of many other useful properties of reach.

\smallskip
\noindent \cite[{\bfseries Theorem 4.8.(2)}]{federer-1959-1}.
If $a \in A$ and $\reach(A, a) > r > 0$, then 
\[\Nor(A, a) = \{\lambda v : \lambda \geq 0, |v| = r,\ \xi(a + v) = a\}.\]

\begin{remark}
Let $\Gamma$ be an embedded $C^1$ closed curve in $\bbR^2$ parameterized by a differentiable injection $\gamma : S^1 \to \bbR^2$ such that $\gamma(S^1) = \Gamma$ and $\dot{\gamma}(t) \neq 0$ for all $t \in S^1$.
Then for any $t \in S^1$, 
\[\Tan(\Gamma, \gamma(t)) = \{\lambda \dot{\gamma}(t) : \lambda \in \bbR\},\]
and for any $v \in \bbR^2$ with $v\cdot \dot{\gamma}(t) = 0$,
\[\Nor(\Gamma, \gamma(t)) = \{\lambda v : \lambda \in \bbR\}.\]
From the above result and Theorem 4.8.(2) of Federer~\cite{federer-1959-1}, we get that if $\reach(\Gamma) > \rho > 0$ then for any $t \in S^1$, $0 < r < \rho$, and for any $v \in \bfU(0, r) \cap \Nor(\Gamma, \gamma(t))$, 
\[\xi(\gamma(t) + v) = \gamma(t).\] 
\end{remark}

\begin{definition}\label{def:opening}
We use the following concepts from mathematical morphology.
Let $E, F\subset \bbR^\dmn$. 
The {\bfseries erosion} of $E$ by $F$ is defined as $E\ominus F = \{x \in E : x + F \subset E\}$ and the {\bfseries opening} of $E$ by $F$ is defined as $E\circ F = (E\ominus F) + F$.
See Figure~\ref{fig:erosionandopening} for an illustration.
\end{definition}

\begin{figure}[htp!]
  \centering 
  \includegraphics[width=1\linewidth]{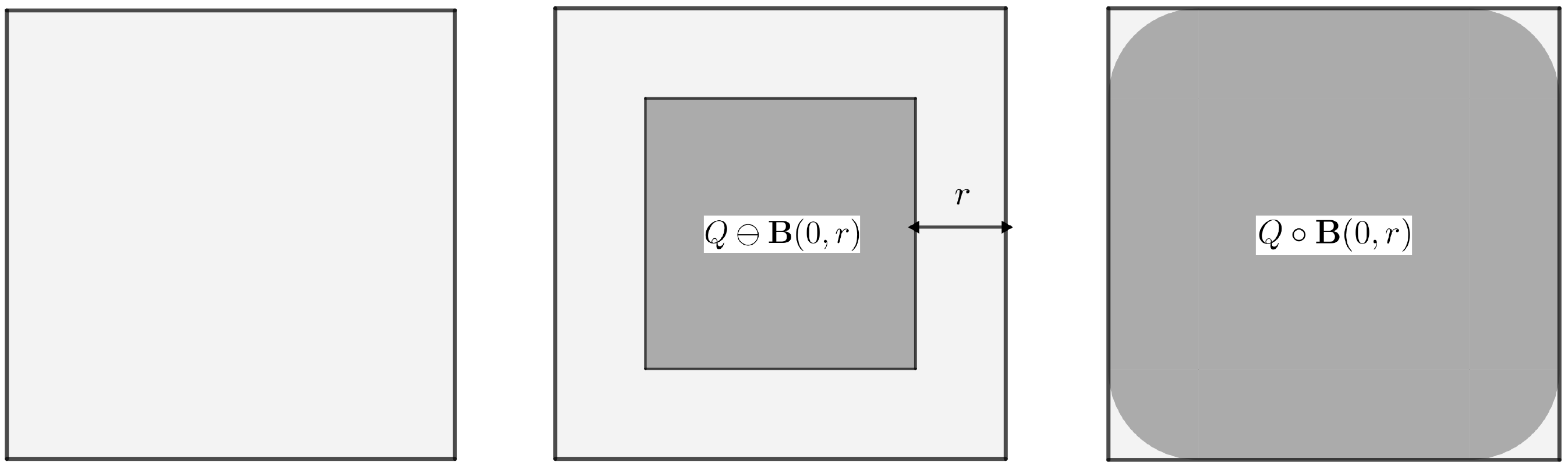}
  \caption{On the left we have a closed cube $Q$.
    The darker region inside $Q$ in the center is $Q \ominus \bfB(0, r)$, the erosion of $Q$ with the closed ball $\bfB(0, r)$.
    On the right we have the opening of $Q$ with $\bfB(0, r)$, which is obtained by taking the Minkowski sum of the dark region from the center figure and $\bfB(0, r)$. 
  \label{fig:erosionandopening}
  }
\end{figure}

\subsection{Currents and the flat norm}

We will be working with $\clL^\dmn$-measurable sets $A \subset \bbR^\dmn$ of finite perimeter with finite $\clL^\dmn$-measure. 
In the language of \emph{currents}, these are sets where the $\dmn$-dimensional current 
\[\bfE^\dmn\myell A\]
is rectifiable, and where its boundary, $\partial(\bfE^\dmn\myell A)$ is also rectifiable.
Currents of this type are {\bfseries integral currents} in $\bbR^\dmn$ and are denoted by $\bfI_n(\bbR^\dmn)$.
Here, 
\[\bfE^\dmn = \clL^\dmn\wedge \bfe_1 \wedge \dots \wedge \bfe_\dmn\]
is the constant $\dmn$-vectorfield on $\bbR^\dmn$ defined by $\bfE^\dmn(x) = e_1 \wedge \dots \wedge e_\dmn$ for all $x \in \bbR^\dmn$, where $\{e_i\}_{i = 1}^\dmn$ forms a standard orthonormal basis of $\bbR^\dmn$.
For an $\clL^\dmn$-measurable subset $A \subset \bbR^\dmn$, the current $\bfE^\dmn\myell A$ is defined so that 
\[\bfE^\dmn(\phi) = \int_A \langle e_1 \wedge \dots \wedge e_\dmn, \phi(x)\rangle\, d\clL^\dmn x \]
for any differential form $\phi$ of degree $\dmn$ and class $\infty$ that is compactly supported in $\bbR^\dmn$.

%Before we give some analytic criteria for a $\clL^\dmn$-measurable subset $A$ to satisfy $\bfE^\dmn\myell A \in \bfI_n(\bbR^\dmn)$, 
We now define the \emph{multi-scale flat norm}, a tool that we will use in Theorem~\ref{thm:fnreach} to ``fix'' certain types of sets $E^c$ that do not have positive reach.

The \emph{flat norm} was originally introduced by Whitney~\cite{whitney2015geometric} and was later used by Federer and Fleming~\cite{federer-1960-1} as a way to endow currents with a metric with the aim of showing the existence of solutions to what is now called \emph{the oriented Plateau's problem}. 
The definition of the flat norm can be found in Federer's book on geometric measure theory~\cite[\S4.1.7.]{federer-1969-1}.
Morgan and Vixie~\cite{MoVi2007} realized that the \emph{$L^1$-Total Variation} 
\[F(u) = \int_{\bbR^\dmn} |\nabla u(x)|\, dx + \lambda \int_{\bbR^\dmn} |f(x) - u(x)|\, dx\]
for functions $u, f : \bbR^\dmn\to \bbR$ of bounded variation is equal to the flat norm on $\partial A$.
In particular, by letting $f$ be the characteristic function of $A$, i.e., $1_A : \bbR^\dmn\to \bbR$ is defined to be $1$ when $x \in A$ and $0$ otherwise, minimizing $F(u)$ over functions of bounded variation is equivalent to minimizing 
\begin{align}\label{eq:fnfiniteperimeter}
E(\Sigma) := \Per(\Sigma) + \lambda\,|\Sigma \Delta A|
\end{align}
over sets of finite perimeter $\Sigma \subset \bbR^n$. 
Here, $\Sigma \Delta A = (\Sigma \setminus A) \cup (A\setminus \Sigma)$. 
This observation motivated their introduction of the following \emph{multiscale flat norm}, which we will now state in the context of both~(\ref{eq:fnfiniteperimeter}) and Federer's book~\cite[\S4.1.24.]{federer-1969-1}. 

\begin{definition}{\bf~\cite{MoVi2007}} \label{def:MFN}
  Let $T$ be an $m$-dimensional current in $\bbR^\dmn$ and let $\lambda \geq 0$. 
  The {\bfseries multiscale flat norm} of $T$ is given by 
  \[\srF_\lambda(T) = \inf \{\bfM(T - \partial S) + \lambda\, \bfM(S)\} \]
  over $(m+1)$-dimensional currents in $\bbR^\dmn$. 

  In the special case that $A \subset \bbR^\dmn$ is an $\clL^\dmn$-measurable set of finite perimeter and $T = \partial (\bfE^\dmn\myell A)$, it follows  that $\srF_\lambda(T)$ is the least number in the set $\{\bfM(T - \partial S) + \lambda\, \bfM(S)\}$ taken over $(m+1)$-dimensional rectifiable currents $S$ in $\bbR^\dmn$ and $m$-dimensional rectifiable currents $T - \partial S$ in $\bbR^\dmn$~\cite[4.2.18.]{federer-1969-1}. 
  In this case, $\srF_\lambda(\partial(\bfE^\dmn\myell A))$ is equal to 
  \begin{align}\label{eq:fn2}
    \inf \{\Per(\Sigma) + \lambda\,|\Sigma \Delta A|\}
  \end{align}
  over sets of finite perimeter $\Sigma \subset \bbR^\dmn$.

  The class of solutions to $\srF_\lambda(T)$ is denoted by $\mathrm{arg}\,\srF_\lambda(T)$, and its members will be denoted as $S_\lambda$. 
  The class of solutions to the optimization problem (\ref{eq:fn2}) is
  \[\{\mathrm{spt}(A - S_\lambda) : S_\lambda \in \mathrm{arg}\,\srF_\lambda(T)\},\]
  and its members will be denoted by $\Sigma_\lambda$. 
\end{definition}
%It turns out that when $T$ is an \emph{integral flat chain}~\cite[\S4.1.24.]{federer-1969-1}, the flat norm of $T$ may be computed by minimizing over $m + 1$ dimensional rectifiable currents $S$ such that $T - \partial S$ is also an $m$-dimensional rectifiable current~\cite[\S4.2.18.]{federer-1969-1}.

We previously studied  \cite{IbKrVi2013} the multiscale flat norm in the simplicial setting where the currents $T$ and $S$ are represented by $m$- and $(m+1)$-dimensional chains, respectively, in a simplicial complex embedded in $\bbR^\dmn$.
We showed that the multiscale simplicial flat norm can be computed efficiently for the setting relevant to this paper---where we want to cover an $\dmn$-dimensional set $E$ in $\bbR^\dmn$---using linear programming.

\subsection{Measure Theoretic Boundary}

To simplify our analysis of the energy $E(\Sigma)$ in Equation (\ref{eq:fnfiniteperimeter}), we introduce the measure theoretic boundary, interior, and exterior. 

Let $E\subset \bbR^n$. 
We define the {\it perimeter} of $E$, as $\Per(\Sigma) := \int |\nabla_{\chi_E}|\,dx$ where $\chi_E : \bbR^n \to \bbR$ is the characteristic function on $E$ defined by $\chi_E(x) = 1$ if $x \in E$ and $0$ otherwise. 
We say a set $\Sigma \subset \bbR^n$ is a {\it set of finite perimeter} if $\Per(\Sigma) < \infty$. 
The structure theorem for sets of finite perimeter tells us that $\Per(\Sigma) = \clH^{n-1}(\partial^\ast \Sigma)$, where $\partial^\ast \Sigma$ is the {\it reduced boundary} of $\Sigma$. 
The reduced boundary is rather complicated to define and difficult to manipulate. 
Instead, we use another theorem which asserts $\partial^\ast \Sigma \subset \partial_\ast \Sigma$ and $\clH^{n-1}(\partial_\ast \Sigma - \partial^\ast \Sigma) = 0$ to conclude that $\Per(\Sigma) = \clH^{n-1}(\partial_\ast \Sigma)$, where $\partial_\ast \Sigma$ denotes the {\it measure theoretic boundary} of $\Sigma$. 
(See \cite[Theorem 2, \S 5.7, Lemma 1 \S 5.8]{evans-1992-1} for more details.) 
We now define measure theoretic boundary, interior, and exterior.

\begin{definition}\label{def: boundary}
Let $A \subset \bbR^n$.
Then the {\it measure theoretic boundary} of $A$ is defined as
\begin{equation}\label{def: mtb}
\partial_\ast A := \left\{x \in \bbR^n : \limsup_{r\to0+} \frac{|\bfU(x, r)\cap A|}{r^n} > 0 \text{ and } \limsup_{r\to0+} \frac{|\bfU(x, r)\cap A^c|}{r^n} > 0\right\}.
\end{equation}
The {\it measure theoretic interior} of $A$ is defined as
\begin{equation}\label{def: mti}
A_\ast^i := \left\{x \in \bbR^n : \limsup_{r\to0+} \frac{|\bfU(x, r)\cap A^c|}{r^n} = 0\right\}.
\end{equation}
The {\it measure theoretic exterior} of $A$ is defined as 
\begin{equation}\label{def: mte}
A_\ast^o := \left\{x \in \bbR^n : \limsup_{r\to0+}\frac{|\bfU(x, r) \cap A|}{r^n} = 0\right\}.
\end{equation}
\end{definition}

\section{Whitney-type Decompositions for Probabilistic Covers}\label{sec:whitney}

%The results
We start with the following observation about how partitioning $E$ into a finite number of pieces $\clR_E = \{R_i\}_{i = 1}^M$ with a uniform lower bound on the mass of each piece and a uniform upper bound on their diameters allows for a lower bound on the probability that the neighborhood of our sampling covers $E$. 

\begin{definition}
We call a partition $\clR_E$ of $E$ a {\bfseries good partition} if for some $\delta > 0$ and $\epsilon > 0$ 
\begin{align}
\diam(R) &\leq \delta\ \mathrm{for\ all}\ R\in \clR_E,\ \mathrm{and}\label{prop:diamupperbound}\\
|R| &\geq \epsilon\ \mathrm{for\ all}\ R\in \clR_E. \label{prop:masslowerbound}
\end{align}
\end{definition}

For simplicity, let us for now assume that $|E| = 1$. 
If we take i.i.d.~samples $\bbX = X_1, \dots, X_N \sim \mathrm{Unif}(E)$ from $E$, then for any sample point $X \in \bbX$ and any partition element $R \in \clR_E$, we have that $P(X \notin R) = 1 - |R|$ and hence $P(\bbX \cap R = \emptyset) = (1 - |R|)^N \leq \exp(-|R|N)$.

Subadditivity then tells us that
\begin{align}\label{eq:probLB}
P(\bfB(\bbX, \delta) \supset E) &\geq 1 - \sum_{R\in \clR}\exp(-|R|N)\\
&\geq 1 - M\exp(-\epsilon N).
\end{align}
%Let $E\subset \bbR^\dmn$ be compact and let $\delta > 0$ and $\epsilon > 0$. 
%To maximize the probability in Equation~(\ref{eq:probability}), we want to be finding partitions of $E$ that minimize
%\begin{align}\label{goal}
%\sum_{R \in \clR(\delta, \epsilon)} e^{-\clL^\dmn(R)N}
%\end{align}
%over all partitions of $E$ whose elements have diameter bounded by $\delta > 0$ and whose volumes are bounded below by $\epsilon$. 

With the goal of bringing geometry into the picture, we aim to derive sufficient and/or necessary geometric conditions on $E$ that will allow us to find a partition $\clR_E$ of $E$ that will satisfy the diameter upper bound property in Equation (\ref{prop:diamupperbound}) and the volume lower bound property in Equation (\ref{prop:masslowerbound}) for suitable values of $\delta$ and $\epsilon$.

Our procedure to find a partition $\clR_E$ of interest initially follows the first parts of the procedure to find the \emph{Whitney decomposition}~\cite[Chapter I\ \S 3.1]{stein1970singular} of $E$.
The Whitney decomposition of $E$ gives us a family $\clQ$ of closed cubes that satisfy
\begin{align}
    &E\subset \bigcup_{Q \in \clQ} Q;\\
    &\interior(Q_i) \cap \interior(Q_j) = \emptyset\ \mathrm{for}\ i\neq j;\ \mathrm{and}\\
    &\diam(Q) \leq \dist(Q, E^c) \leq 4\,\diam(Q) \ \mathrm{for\ all}\ Q \in \clQ,
\end{align}
where $\interior(Q)$ denotes the interior of the cube $Q$ and $\diam(Q)$ denotes the diameter of $Q$.

While following Whitney's procedure, we need to make sure that the diameter of each cube is not larger than $\delta$ and that the volume of each cube is larger than $\epsilon$, i.e., it does not get too small.  
Obtaining the cubes in $\clQ$ that satisfy the first criterion in Equation (\ref{prop:diamupperbound}) is easy since we may further subdivide any large cubes into smaller ones. 
The second criterion, however, is a bit harder to satisfy since the diameters, and hence volumes, of the cubes get arbitrarily small as they approach the boundary of $E$.
 
%The next thing we must do is to obtain the cubes $\clQ_\delta \subset \clQ$ whose diameters are larger than or equal to $\delta > 0$.
Now, for $\delta > 0$, let $\clQ_\delta$ be the collection of all cubes in $\clQ$ whose diameters are larger than or equal to $\delta$. 
Although $\cup_{Q \in \clQ_\delta} Q \supset \{x \in E : \dist(x, E^c) \geq 2\delta\}$,
we do not have any control over 
\[\sup_{y \in E}\dist(y, \cup_{Q \in \clQ_\delta} Q)\]
(see Figure~\ref{fig:dumbell_whitney}).
This implies that we might have to consider substantially large neighborhoods around each cube in $Q_\delta$ in order to cover an arbitrary compact subset $E$ of $\bbR^\dmn$.

\begin{figure}[htp!]
  \centering \includegraphics[width=.8\linewidth]{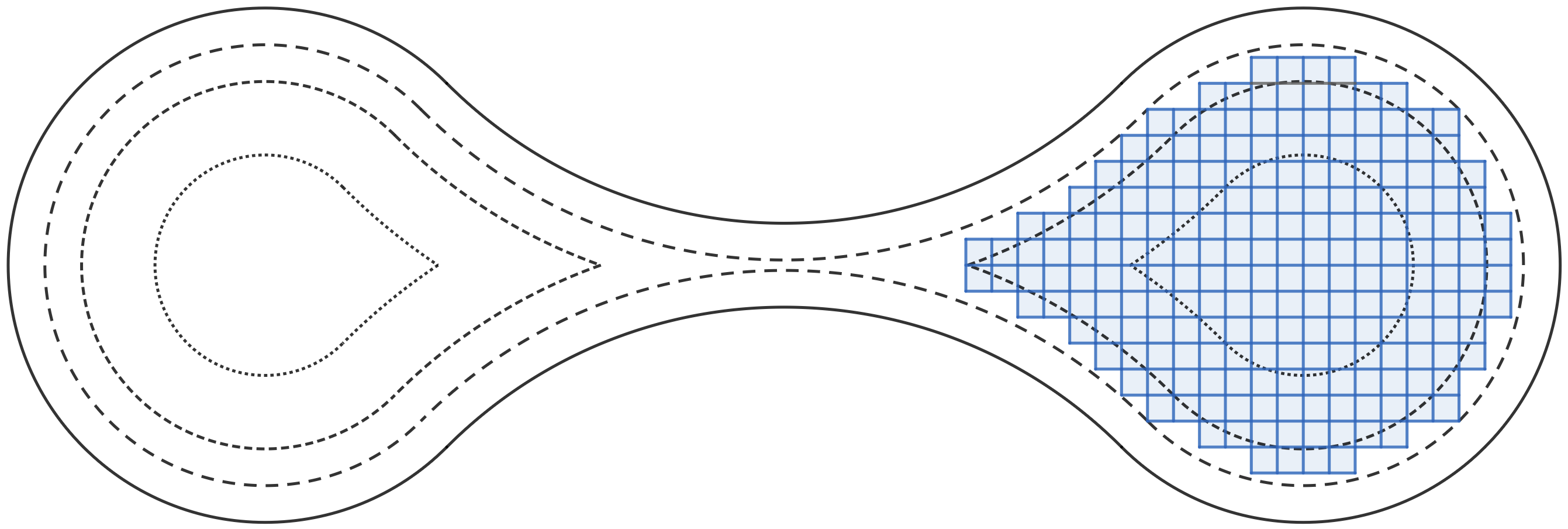}
  \caption{The set $E$ is everything inside the dumbbell shape. 
  Everything inside first set of dotted lines is at least $\delta$ away from $E^c$. 
  Similarly, inside the second set of dotted lines is at least $2\delta$, and everything inside the third set of dotted lines is at least $2^2\delta$ away from $E^c$. 
  Half of the collection of cubes of diameter $\delta$ in $\clQ_\delta$ are shown on the right. 
  These cubes will cover everything inside $E$ that is at least $2\delta$ away from $E^c$.
  Notice that elongating the neck region would not change the cubes, however, the largest distance $\sup_{y \in E}\dist(y, \cup_{Q\in \clQ_\delta}Q)$ from points in $E$ to the cubes becomes larger.}
  \label{fig:dumbell_whitney}
\end{figure}

We seek sufficient conditions on $E$ which guarantee that $\sup_{y \in E} \dist(y, \clQ_\delta)$ is small enough.  
In particular, if $E$ is equal to the \emph{opening} (see Definition~\ref{def:opening}) of $E$ by the open ball of radius $\delta$, i.e.,
\[E = E \circ \bfU(0, \delta),\]
we may do the following. 
First, consider the lattice $\ell \bbZ^\dmn$ that defines a family of cubes $\Delta_\ell$, and choose $\ell$ so that the diameter of each cube in $\Delta_\ell$ is $\delta$. 
We will know that the union of the $\delta$-neighborhoods of each cube in $\Delta_\ell$ that intersects the \emph{erosion} (see Definition~\ref{def:opening}) of $E$ by $\bfB(0, \delta)$ will cover $E$.

Although these fattened cubes are not pairwise disjoint anymore, their diameters are bounded above by $3\delta$ and we are easily able to turn them into a partition $\clR_E$ of $E$ with the same bound on their diameters, and with a lower bound of $\delta^\dmn/(\dmn^{\dmn/2})$ on their volumes.
This argument is detailed in the proof of Lemma~\ref{lem:opening}, where we assume a related condition called \emph{reach}, a concept that was introduced in Federer's seminal paper~\cite{federer-1959-1}. 
This relationship is given by the following fact, which follows from~\cite[Lemma 4.8]{rataj2019curvature}.
\emph{If $\reach(E^c) > \rho > 0$, then $E = E\circ \bfU(0, \rho)$.} 
The converse of this fact is, however, not true as can be seen in the following example. %\cref{ex:reachandopening}.

\begin{example}\label{ex:reachandopening}
  Although $\reach(E^c) = \delta > 0$ implies that $E = E\circ \bfU(0, \delta)$, the converse is not true. 
  For example, let $x, y \in \bbR^2$ such that $\norm{x - y} = 2R$. 
  Then for any $\delta > R$, $E := \bfU(x, \delta) \cup \bfU(y, \delta)$ has $\reach(E^c) < \delta$ but $E = E\circ \bfU(0, \delta)$.
\end{example}

\begin{remark}
  Although the sufficiency of $E = E\circ \bfU(0)$ implies that positive reach is also sufficient, the reach will still serve an important role in this paper due to the reach properties of flat norm minimizers in $\bbR^2$.
\end{remark}
 
\begin{lemma}\label{lem:opening}
	Let $E$ be a bounded open subset of $\bbR^\dmn$.
	If
	\[\reach(E^c) > \rho > 0,\] 
	then for any $0 < \delta < \rho$, there exists a finite partition $\clR_E$ of $E$ such that   
	\begin{align}
		|R| \geq \frac{\delta^\dmn}{\dmn^{\dmn/2}}\quad \mathrm{and} \quad \mathrm{diam}(R) \leq 3\delta\quad \mathrm{for\ all}\ R \in \clR_E.
	\end{align}
\end{lemma}
\begin{proof}
  %As explained in the introduction,
  To use the Whitney decomposition on $E$, we first note that since $\reach(E^c) > \rho > \delta > 0$, we have $E = E\circ \bfU(0, \delta)$. 
  By definition, this means that $E$ is the $\delta$-neighborhood of $E^\ast = \{x \in E : \dist(x, \partial E) \geq \delta\}$, the points in $E$ that are at least $\delta$ away from the boundary, i.e., $\bfB(E^\ast, \delta) = E$.
  This implies that if we are able to find a family $\clQ$ of cubes of side lengths $\ell \geq 0$ which covers $E^\ast$, then the collection of all neighborhoods $\bfB(\clQ, \delta) := \{\bfB(Q, \delta)\}_{Q \in \clQ}$ will be a cover of $E$.
  Since the neighborhoods in $\bfB(\clQ, \delta)$ may not be pairwise disjoint, if in addition, our family $\clQ$ satisfies 
  \begin{align*}
    Q &\subset E \quad \mathrm{and}\quad \interior(Q)\cap \interior(Q') = \emptyset\quad \mathrm{for\ all}\quad Q\neq Q' \in \clQ,
  \end{align*}
  then the collection of $\bfB(\clQ, \delta)$ may be used to construct the desired partition $\clR_E$ of $E$ such that its partitioning elements have a uniform upper bound in diameter and lower bound in volume.
  In particular, $\clR_E$ will satisfy
  \begin{align*}
    |R| \geq \ell^\dmn\quad \mathrm{and}\quad \diam(R) \leq \sqrt{n}\ell + 2\delta \quad \mathrm{for\ all}\quad R\in \clR_E.
  \end{align*}

  Consider an arbitrary ordering of the cubes in $\clQ = \{Q_i\}_{i = 1}^M$. 
  By way of induction, we first define $R_1$ to be $\bfB(Q_1, \delta)\cap E$ except for the region of $E$ that is contained by the rest of the cubes, i.e.,
  \[R_1 = [\bfB(Q_1, \delta)\cap E]\setminus \bigcup_{j = 2}^M Q_j.\]
  Then, for $k = 2, 3, \dots, M$, we similarly define $R_k$ to be $\bfB(Q_k, \delta)\cap E$ except for the region of $E$ that is contained by the rest of the cubes in $\{Q_j\}_{j = k + 1}^M$ or the previously assigned partitions in $\{R_j\}_{j = 1}^{k-1}$, i.e., let
  \[R_k = [\bfB(Q_k, \delta) \cap E] \setminus \left[\bigcup_{j = 1}^{k-1}R_j \cup \bigcup_{j = k + 1}^M Q_j\right].\]

  The method employed by the Whitney decomposition of $E$ will provide us with $\clR_E$. 
  Simply speaking however, for any $\ell > 0$, the lattice $\ell\,\bbZ^\dmn$ defines a family of cubes $\Delta_\ell$ of side length $\ell$, where the diameter of each cube in $\Delta_\ell$ equals $\sqrt{\dmn}\ell$.
  Letting $\ell = \delta/\sqrt{\dmn}$, any cube $Q$ in 
  \[
  \clQ = \{Q \in \Delta_\ell : Q\cap E^\ast \neq \emptyset\}
  \]
  will be contained in $E$ since $\diam(Q) \leq \delta$. 
  Using such a $\clQ$, we find that our constructed partition $\clR_E$ of $E$ satisfies 
  \[
  |R| \geq \frac{\delta^\dmn}{\dmn^{\dmn/2}}, \text{ and } \diam(R) \leq 3\delta \quad \text{for}\quad R\in \clR_E.
  \]
\end{proof}

\begin{proposition}\label{thm:opening2}
	Let $E$ be a subset of $\bbR^\dmn$ such that
	\[\sup_{x \in E} \dist(x, E^c) = \rho^\ast > 0.\] 
	Then for any $0 < \delta < \rho^\ast$, there exists a finite partition $\clR_E$ of $E$ such that   
	\begin{align}
	  |R| \geq \frac{\delta^\dmn}{\dmn^{\dmn/2}}\quad \mathrm{and} \quad \mathrm{diam}(R) \leq \delta + 2\eta_\delta \quad \mathrm{for\ all}\ R \in \clR_E,
	\end{align}
	where 
	\[\eta_\delta = \sup_{y \in E}\dist(y, E\ominus \bfB(0, \delta)).\]
\end{proposition}

\begin{proof}
  We get this result with the same construction we used in the proof of Lemma~\ref{lem:opening} modulo the collection of all neighborhoods $\bfB(\clQ, \eta_\delta) = \{\bfB(Q, \eta_\delta)\}_{Q \in \clQ}$, which we now need to cover $E$.
\end{proof}

\begin{remark}
  For any $\delta > 0$ satisfying the assumption of Lemma~\ref{lem:opening}, i.e., $\delta \in (0, \reach(E^c))$, we have that $\eta_\delta = \delta$. 
  In general, for the assumptions of Proposition~\ref{thm:opening2}, we only know that $\eta_\delta \geq \delta$. 
\end{remark}

We now apply Lemma~\ref{lem:opening} to answer the main question on the covering probability in Equation (\ref{eq:probability}).

\begin{theorem}\label{thm:reach}
  If $E$ is a bounded open subset of $\bbR^\dmn$ with $\reach(E^c) > \rho > 0$, then for any $\delta \in (0, \rho)$ there exists a natural number $M = M(\delta, E)$ such that for any independently and identically distributed sampling $\bbX_N = (X_i)_{i = 1}^N$ of $E$, 
  \begin{align}\label{eq:problowerbound}
    P(\bfB(\bbX_N, 3\delta) \supset E) \geq 1 - M\exp\left(\frac{-\delta^\dmn}{\dmn^{\dmn/2}|E|} N\right).
  \end{align}
\end{theorem}

\begin{proof}
  We prove the lower bound in Equation~(\ref{eq:problowerbound}) by applying Lemma~\ref{lem:opening} and elementary properties of probability theory. 

  Since $\delta < \rho$, and $\reach(E^c) > \rho$ implies $\reach(\partial E) > \rho$, by Lemma~\ref{lem:opening} we know that there exists a finite partition $\clR_E$ of $E$ such that 
  \[ 
  |R| \geq \frac{\delta^\dmn}{\dmn^{\dmn/2}}\qquad \text{and} \quad \diam(R) \leq 3\delta\quad \text{ for all } R\in \clR_E.
  \]
  Let $M$ be the cardinality of the partition $\clR_E$. 
  As in the introduction, for any sample point $X \in \bbX$ and partition element $R \in \clR_E$, we have that $P\{X \notin R\} = 1 - |R|/|E|.$
  Thus,  
  \begin{align*}
    P\{\bbX_N \cap R = \emptyset\} &= (1 - |R|/|E|)^N \\
    &\leq \exp\left(-\frac{|R|}{|E|}N\right)
  \end{align*}
  and hence
  \begin{align*}
    P\left(\bigcup_{R\in \clR}\{\bbX\cap R\} = \emptyset \right) &\leq \sum_{R\in\clR}P(\bbX\cap R = \emptyset)\\
    &\leq \sum_{R\in\clR} \exp\left(-\frac{|R|}{|E|}N\right)\\
    &\leq M\exp\left(-\frac{\inf_{R\in \clR} |R|}{|E|}N\right)\\
    &\leq M\exp\left(-\frac{\delta^\dmn}{\dmn^{\dmn/2}|E|}N\right).
  \end{align*}
  Finally, since $\diam(R) \leq 3\delta$, the closed ball $\bfB(X, 3\delta)$ around the realization of any sample $X \in \bbX_N$ will contain the partition element $R_X \in \clR_E$ that was hit by $X$.  
  Therefore, 
  \begin{align*}
    P(\bfB(\bbX_N, 3\delta) \supset E) &\geq 1 - M\exp\left(\frac{-\delta^\dmn}{\dmn^{\dmn/2}|E|}N\right).
  \end{align*}
\end{proof}

\section{Beyond Whitney}\label{sec:beyondwhitney}

It turns out that there are plenty of compact subsets $E$ which \emph{do} have nice partitionings $\clR_E$, but \emph{do not} have complements of positive reach nor are equal to openings of themselves with some $\bfU(0, \delta)$.  
For an example in Figure~\ref{fig:EminusA}, we remove a very small $\eta > 0$ amount of volume $A$ from the closed ball $F = \bfB(0, R)$, in such a way that $E^c = (F\setminus A)^c$ will have $0$ reach and $E\neq E\circ \bfU(0, \delta)$ for any $\delta > 0$.

We present Corollary~\ref{cor:UminusA} which states that one is able to get a similar lower bound to the probabilistic covering of sets $E$ that are obtained by removing a small volume from sets of positive reach.

\begin{figure}[htp!]
  \centering \includegraphics[width=.2\linewidth]{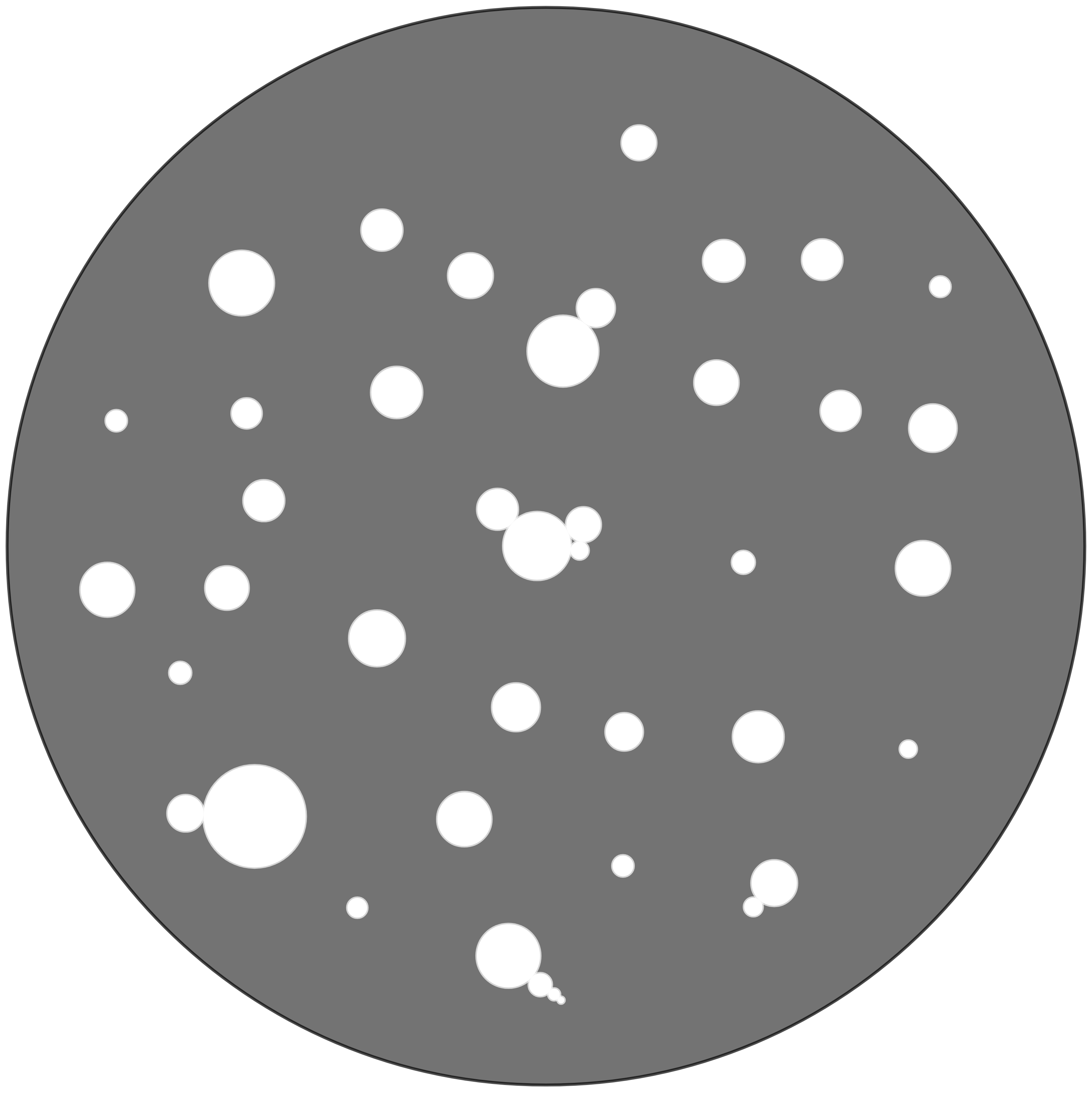}
  \caption{The set $E = \bfB(0, R)\setminus A$ such that $|A| < \epsilon$, $\reach(E^c) = 0$, and $E \neq E\circ \bfU(0, \delta)$ for any $\delta > 0$.    \label{fig:EminusA}
  }
\end{figure}

\begin{corollary}\label{cor:UminusA}
  Let $U$ be a bounded open subset of $\bbR^\dmn$ such that $\reach(U^c) > \rho > 0$, let $\delta \in (0, \rho)$, and let $A \subset U$ be such that $|A| < \delta^\dmn n^{-n/2}$.
  For $E = U\setminus A$, there exists a partition $\clR_E$ of $E$ such that for any finite i.i.d.~sampling $\bbX$ of $E$,
  \begin{align}\label{eq:problowerbound2}
    P(\bfB(\bbX, 3\delta) \supset E) \geq 1 - M\exp\left(\frac{-(\delta^\dmn n^{-n/2} - |A|)}{|E|} N\right),
  \end{align}
  where $N$ denotes the number of samples and $M$ denotes the cardinality of the partition $\clR_E$ of $E$.
\end{corollary}

\begin{proof}
  Since $\reach(U^c) > \rho > 0$, and $0 < \delta < \rho$, by Lemma~\ref{lem:opening}, there exists a partition $\clR_U$ of $U$ such that
  \[|R| \geq \frac{\delta^{\dmn}}{\dmn^{\dmn/2}} \quad \text{ and } \quad \diam(R) \leq 3\delta \quad \text{ for all } R\in \clR_U.\]
  Since $E \subset U$, we may construct a partition $\clR_E$ of $E$ with 
  \[\clR_E = \{R\cap E : R \in \clR_U \}.\]
  By noticing that $\inf_{R' \in \clR_E} \clL^\dmn(R') \geq \inf_{R \in \clR_U}  |R| - |A| > 0$, we may apply the argument in Theorem~\ref{thm:reach} to $\clR_E$ and conclude that 
  \begin{align}\label{eq:problowerbound3}
    P(\bfB(\bbX, 3\delta) \supset E) \geq 1 - M\exp\left(\frac{-(\delta^\dmn n^{-n/2} - |A|)}{|E|} N\right),
  \end{align}
  where $N$ denotes the number of samples and $M$ denotes the cardinality of the partition $\clR_U$.
\end{proof}

\subsection{The multiscale flat norm and probabilistic covers in $\bbR^2$}\label{sec:flatnormR2}

Let $E$ be a set of finite perimeter in $\bbR^2$. 
Corollary \ref{cor:UminusA} states that if we know $E = U\setminus A$ where $U^c$ has positive reach, then we are in good shape.
However, if $E$ does not necessarily have this form, can we still say something? 
In this section we show that in $\bbR^2$, we have a method that allows us to obtain a lower bound on the probability that we ``almost'' cover all of $E$. 
This method involves the \emph{multiscale flat norm} (see Definition \ref{def:MFN}).

Applying the multiscale flat norm $\srF_\lambda$ with scale $\lambda > 0$ to $\partial E$ gives us a ``denoised'' version $\partial E - \partial S_\lambda$ of $E$ when $S_\lambda \in \arg \srF_\lambda(\partial E)$. 
In Theorem~\ref{thm: flatNormReach}, we prove that for sets $E$ of finite perimeter in the plane, we have $\reach(E_\lambda^c) > 1/(5\lambda)$ for $E_\lambda := E - S_\lambda$.
Since Lemma~\ref{lem:opening} applies to sets of positive reach, we investigate when the partition $\clR_{E_\lambda}$ of $E_\lambda$ may be used to give a good partition of $E$.

In general we cannot guarantee that the partition $\clR_{E_\lambda}$ can always be used to give a good partition of $E$. 
However, the following Theorem~\ref{prop:flatnormsufficient} does give us a sufficient condition for when the partition $\clR_{E_\lambda}$ of $E_\lambda$ may be used to give a \emph{good $\alpha$-almost partition} of $E$ (see Definition \ref{def:goodalphaalmostpartition}).
Intuitively, a \emph{good $\alpha$-almost partition} of $E$ is a ``large'' subset of $E$ with a good partition. 
The existence of a good $\alpha$-almost partition $E$ will be used to provide a lower bound for the probability that $\bfB(\bbX, 3\delta)$ covers a ``large'' subset of $E$.  

\begin{definition}\label{def:goodalphaalmostpartition}
  Let $E\subset \bbR^\dmn$ and $\alpha \in \bbR$.
  A partition $\clR_A$ of a set $A$ is called an $\boldsymbol{\alpha}$-{\bfseries almost partition} of $E$ if
  \[A \subset E \quad \text{ and } \quad |E\setminus A| \,\geq\, (1 - \alpha)|E| \, \geq 0.\]
  We say that $\clR_A$ is a {\bfseries good $\boldsymbol{\alpha}$}-{\bfseries almost partition} of $E$ if $\clR_A$ is a good partition of $A$ and if $\clR_A$ is an $\alpha$-almost partition of $E$.
\end{definition}

Before we state the main result, we consider the case when $\lambda$ is small enough so that $E_\lambda = 0$. 
Since we cannot use the empty set to form any good almost partition of $E$, we must require that $\lambda$ is large enough. 
In particular, we require that $\lambda > \Lambda_E$, where $\Lambda_E = \sup\{\lambda > 0 : E \in \arg \srF_{\lambda}(\partial E)\}$. 
Our sufficient condition says that for any $\lambda > \Lambda_E$ and $\delta > 0$ such that 
\[|S_\lambda| < \frac{\delta^2}{2}\quad \text{ and } \quad 0 < \delta < \frac{1}{5\lambda}\]
the partitions $\clR_{E_\lambda} = \clR^\delta_{E_\lambda}$ of $E - S_\lambda$ may indeed be used as to obtain a good $\alpha$-almost partition of $E$ for a particular $\alpha$ that depends on $\delta$.
The proof of Theorem~\ref{prop:flatnormsufficient} uses Lemmas~\ref{lem: 16},~\ref{lem:tangents-match},~\ref{lem: equal-tangents}, and Theorems~\ref{thm:one-circle}, and~\ref{thm: flatNormReach} from our previous work~\cite{alvarado2017lower}.
We reproduce these results here for the sake of completeness.
Recall the notion of outer-normal map $\mathbf{v}_\epsilon$ given in Definition \ref{def:outinormal}.

\begin{lemma}\label{lem: 16}
Let $\gamma : \mathbb{S}^1\to \bbR^2$ be a $C^1$ isometric embedding and suppose $n$ is differentiable at $s\in \mathbb{S}^1$ such that $|N'(s)| < K$. Then if $0 < \epsilon < 1/K$ then $\Tan( \Gamma, \gamma(s)) = \Tan(\Gamma_\epsilon, \mathbf{v}_\epsilon(s))$. 
\end{lemma}

\begin{proof}
Since we are trying to show equality of our tangent spaces, all we must show is that $\mathbf{v}'(s) = c\gamma'(s)$ for some $c \neq 0$. 

Let $N$ be differentiable at $s \in \mathbb{S}^1$. 
Since $\norm{N} = 1$ on $\mathbb{S}^1$, $\langle N'(s),N(s)\rangle = 0$.  
Hence, since $\langle \gamma'(s), N(s)\rangle = 0$ there exists an $r\in\mathbb{R}$ for which $N'(s) = r\gamma'(s)$.
Thus $\mathbf{v}'(s) = \gamma'(s)(1 + r\epsilon)$. 
We must therefore show that $1 + r\epsilon \neq 0$; which reduces to showing that $\epsilon \neq 1/r$. 

Since $|N'(s)| < K$, we get that $|r| < K$. 
Moreover since $0 < \epsilon < 1/K$, we get that $\epsilon < 1/|r|$ and hence 
$-1/|r| < \epsilon < 1/|r|$; implying that $\epsilon \neq -1/r$. 
\end{proof}

\begin{lemma}\label{lem:tangents-match}
If $\Gamma$ is a $C^{1,1}$ embedding of $\mathbb{S}^1$ into $\mathbb{R}^2$ and $\epsilon < 1/K$, then $\Tan(\Gamma, \gamma(s)) = \Tan(\Gamma_\epsilon, \mathbf{v}_\epsilon(s))$ for all $s \in \mathbb{S}^1$. 
\end{lemma}

\begin{proof}
Since we will prove a local condition on $\Gamma$ and $\Gamma_\epsilon$, we may instead investigate $\mathbf{v}_\epsilon\vert_E$ for sufficiently small open sets $E$ about points in $\mathbb{S}^1$. That is, $\Gamma\vert_E$ and $\Gamma_\epsilon\vert_E$ instead of $\Gamma$ and $\Gamma_\epsilon$ respectively. We will first argue that $\Gamma_\epsilon\vert_E$ is a $C^1$ hypersurface, and then show that $\Tan(\Gamma\vert_E, \gamma(s)) = \Tan(\Gamma_\epsilon\vert_E, \mathbf{v}_\epsilon(s))$ for all $s \in E$. 

Since $\Gamma\vert_E$ is $C^{1,1}$, reach$(\Gamma\vert_E) >$ 0 \cite{lucas1957submanifolds}; so by a result of Krantz and Parks \cite{krantz1981distance} we get that the signed distance function $\delta(x)$ to $\Gamma\vert_E$ is a $C^1$ function on the open $\epsilon$-neighborhood $\bfU(\Gamma\vert_E, \epsilon)$ of $\Gamma\vert_E$. Hence the $\epsilon$-level-set $L_\epsilon (\delta)$ is also $C^1$; and since $L_\epsilon(\delta) = \Gamma_\epsilon\vert_E$, we get that $\Gamma_\epsilon\vert_E$ is also $C^1$. 

Now since $\Gamma$ is $C^{1,1}$, $n$ is Lipschitz and therefore differentiable a.e.~in $\mathbb{S}^1$. 
Hence by Lemma~\ref{lem: 16},
$\Tan(\Gamma, \gamma(s)) = \Tan(\Gamma_\epsilon, \alpha(s))$ for a.e.~$s\in \mathbb{S}^1$. 
Therefore if we pick $x \in \mathbb{S}^1$ such that $N$ is not differentiable at $x$, there exists a sequence $\{x_i\}_{i=1}^\infty$ in $E$ converging to $x$ for which $\Tan(\Gamma, \gamma(x_i)) = \Tan(\Gamma_\epsilon, \mathbf{v}_\epsilon(x_i))$ for all $i$. 
This, together with the fact that $\Gamma_\epsilon\vert_E$ is $C^1$, implies that $T(\Gamma_\epsilon, \mathbf{v}_\epsilon(x)) = \Tan(\Gamma, \gamma(x))$.
\end{proof}

\begin{lemma}\label{lem: equal-tangents}
  Let $\gamma: \mathbb{S}^1 \to \bbR^2$ and let $\Gamma = \gamma(\mathbb{S}^1)$. If $\gamma$ is $C^{1,1}$, then there exist distinct points $v, v' \in \mathbb{S}^1$ such that $\mathbf{v}_\rho(v) = \mathbf{v}_\rho(v')$ where $\reach(\Gamma) > \rho > 0$. 
  Moreover, $\Tan(\Gamma, \gamma(v)) = \Tan(\Gamma, \gamma(v'))$.
\end{lemma}

\begin{proof}
  First and foremost, we have to set everything up. Let $\Gamma$ have positive reach $\rho$. Since $\Gamma$ is compact, there exist distinct $v, v' \in \mathbb{S}^1$ for which $\mathbf{v}_\rho(v) = \mathbf{v}_\rho(v')$ or $\bar{\mathbf{v}}_\rho(v) = \bar{\mathbf{v}}_\rho(v')$. Without loss of generality, let us assume $\mathbf{v}_\rho(v) = \mathbf{v}_\rho(v')$. 
  Therefore, for the following arguments, we may use the distance function on subsets $U\subset \Gamma$
\begin{align*}
 \delta_U(x) = \mathrm{dist}(x, U).   
\end{align*}

For any $\epsilon > 0$, consider the two intervals $V_\epsilon = (v - \epsilon, v + \epsilon)$, and $V'_\epsilon = (v' - \epsilon, v' + \epsilon)$ around $v$ and $v'$, and the subsets $\Gamma_{v,\epsilon} := \mathbf{v}_\rho(V_\epsilon)$ and $\Gamma_{v',\epsilon} := \mathbf{v}_\rho(V'_\epsilon)$ of $\Gamma$.
For notational convenience, let $\delta_v := \delta_{\Gamma_{v,\epsilon}}$, $\delta_{v'} := \delta_{\Gamma_{v',\epsilon}}$.

Notice that by the equality of injectivity radius and reach, for any $\sigma \leq \rho$ and any $x \in \mathbf{v}[(0, \sigma)\times V_\epsilon]$ there exists an $s \in \Gamma_v$ such that $\norm{x-s} = \delta(x)$, and hence $\delta_{v}(x) = \delta(x)$. Similarly, for any $x \in \mathbf{v}[(0, \sigma)\times V'_\epsilon]$ there exists an $s' \in \Gamma_{v'}$ for which $\norm{x - s'} = \delta(x)$ and hence $\delta_{v'}(x) = \delta(x)$.

By way of contradiction, suppose the tangent spaces are not equal. 
Therefore, $E_\epsilon := \mathbf{v}[[0,\rho]\times V_\epsilon]\cap \mathbf{v}[[0,\rho]\times V'_\epsilon]$  will have a non-empty interior, and hence there exists $\nu_2 > 0$ such that given any $0 < \epsilon < \nu_1$, for any point $x \in \mathbf{v}_\rho(\mathbf{v}|_{V}^{-1}(E_\epsilon))$, $\delta_{v}(x) > \delta_{v'}(x)$, and for any point $x \in \mathbf{v}_\rho(\mathbf{v}|_{V'}^{-1}(E_\epsilon))$, $\delta_{v}(x) < \delta_{v'}(x)$. Therefore, if we define a new function $\tilde{\delta}: \mathbb{R}^2 \to \mathbb{R}$ as 
 \begin{equation*}
  \tilde{\delta}(x) := \delta_{v}(x) - \delta_{v'}(x),   
 \end{equation*}
we find that $\tilde{\delta}(x)$ is positive whenever $x \in \mathbf{v}_\rho(\mathbf{v}|_{V}^{-1}(E_\epsilon))$ and negative whenever $x \in \mathbf{v}_\rho(\mathbf{v}|_{V'}^{-1}(E_\epsilon))$.

Now, let $\hat{x} \in \mathbf{v}_\rho(\mathbf{v}|_{V}^{-1}(E_\epsilon))$ and $\hat{x}' \in \mathbf{v}_\rho(\mathbf{v}|_{V'}^{-1}(E_\epsilon))$. 
Since $\tilde{\delta}$ is continuous, there exists $r \in (0,1)$ such that for $z := r\hat{x} + (1-r)\hat{x}'$, $\tilde{\delta}(z) = 0$, and hence $\delta_v(z) = \delta_{v'}(z)$.

 We let $s$ and $s'$ in $\Gamma_v$ and $\Gamma_{v'}$, respectively, such that $\delta_{v'}(z) = \norm{s' - z} = \delta(z) = \norm{s - z} = \delta_{v}(z) < \rho$. 
 Since there exists a $\nu_2 > 0$ such that for any $0 < \epsilon < \nu_2$, $\Gamma_{v,\epsilon} \cap \Gamma_{v', \epsilon} = \emptyset$, by letting $\epsilon < \min{\{\nu_1, \nu_2\}}$, we get a contradiction to $z$ having a unique closest point in $\Gamma$. 
\end{proof}

\begin{remark}
  At first glance, the proof might appear to not use the fact that
  curvature of $\Gamma$ is uniformly bounded above by something
  strictly less than $1/\rho$. However, we do so
    indeed---we need this condition in order for
  $\Tan(\alpha(V_\epsilon, \rho), \alpha(v, \rho))$ and
  $\Tan(\alpha(V'_\epsilon, \rho), \alpha(v', \rho))$ to be tangent
  lines for small enough $\epsilon$. This is implied by
  Lemma~\ref{lem:tangents-match}.
\end{remark}

%\begin{figure}[htp!]
%  \centering
%\scalebox{1}{\input{../Figures/reduced-reach.pdf_t}}  
%  \caption{By expanding $\Gamma$ with the correct normal map by $\rho$ amount, we assume that the tangent lines are not equal, and then show that there is a point $z$ inside $E_\epsilon$, and two corresponding points, $s$ and $s'$ in $\Gamma_{v}$ and $\Gamma_{v'}$ for which $|z - s'| = |z - s| = \delta(z)$ and is less than $\rho$ contradicting the fact that $\rho$ is the reach.}
%  \label{fig:neck1}
%\end{figure}

\begin{lemma}
Given $\Gamma = T - \partial S$, if there exists a current $S^\ast$ for which $\mathbf{M}(\Gamma) > \mathbf{M}(\Gamma - \partial S^\ast) + \lambda\mathbf{M}(S^\ast)$ then $\mathbf{M}(\Gamma) + \lambda\mathbf{M}(S) > \mathbf{M}(T - \partial(S + S^\ast)) + \lambda\mathbf{M}(S + S^\ast)$.
\label{thm:lem-1}
\end{lemma}

\begin{proof}
\begin{align*}
\mathbf{M}(\Gamma) + \lambda\mathbf{M}(S) &= \mathbf{M}(T - \partial S) + \lambda\mathbf{M}(S)\\
&> \mathbf{M}(\Gamma - \partial S^\ast) + \lambda \mathbf{M}(S^\ast) + \lambda\mathbf{M}(S)\\
&\geq \mathbf{M}(\Gamma - \partial S^\ast) + \lambda\mathbf{M}(S^\ast + S)\\
&\geq \mathbf{M}(T-\partial S - \partial S^\ast) + \lambda\mathbf{M}(S^\ast + S)\\
&\geq \mathbf{M}(T - \partial (S + S^\ast)) + \lambda\mathbf{M}(S^\ast + S)
\end{align*}
\end{proof}

\begin{theorem}
\label{thm:one-circle} 
If $\Gamma_\lambda$ is a $C^{1,1}$ isometric embedding of $\mathbb{S}^1$ into $\mathbb{R}^2$, then the reach of $\Gamma_\lambda$ is bounded below by $C/\lambda$, where $C \approx 0.22$.
\end{theorem}

\begin{proof} 
Let $\Gamma$ be a $C^{1,1}$ isometric embedding of $\mathbb{S}^1$ into $\mathbb{R}^2$. 
Suppose by way of contradiction that $\reach(\Gamma) < C/\lambda$. 
By a comparison argument, we will show that there is a local perturbation of our current $\Gamma$ into $\Gamma^\ast := \Gamma - \partial S^\ast$ for which 
\begin{equation*}
\mathbf{M}(\Gamma) > \mathbf{M}(\Gamma^\ast) + \lambda\mathbf{M}(S^\ast),
\end{equation*}
to then conclude, by Lemma~\ref{thm:lem-1}, that $\Gamma$ was not a minimizer in the first place.

\begin{enumerate}[label={\bfseries Step \arabic{enumi}.}]
\item %{\bf Step 1.}
First, define $\mathbf{v}_\epsilon$ and $\bar{\mathbf{v}_\epsilon}$ to be the outer and inner-normal map of $\Gamma$ as specified in Definition \ref{def:outinormal}.
Similarly, we let $\rho$ be the injectivity radius of the inner and outer normal map of $\Gamma$; that is, the supremal $\epsilon$ for which both $\mathbf{v}_\epsilon$ and $\bar{\mathbf{v}}_\epsilon$ are injective on $\mathbb{S}^1$. 

Since $\reach(\Gamma) > 0$ \cite{lucas1957submanifolds} and $\Gamma$ is compact, we may assume that $\rho > 0$ and that there exists distinct $v, v' \in \mathbb{S}^1$ for which $\mathbf{v}_\rho(v) = \mathbf{v}_\rho(v')$ or $\bar{\mathbf{v}}_\rho(v) = \bar{\mathbf{v}}_\rho(v')$. 
Without loss of generality, for the remainder of the proof we will assume that $\mathbf{v}_\rho(v) = \mathbf{v}_\rho(v')$.

\item %{\bf Step 2.}
In this step, for any $-\lambda^{-1} \leq t \leq  \lambda^{-1}$, we will construct a current $S^\ast(t)$. 
We will then find an optimal $t'$ and then define $S^\ast := S^\ast(t')$.

Without loss of generality, we will work under the translation $\Gamma - \gamma(v)$ and the rotation for which $\Tan(\Gamma, \gamma(v)) - \gamma(v)$ is equal to a coordinate axis and $\Tan(\Gamma, \gamma(v')) - \gamma(v')$ is equal to $\mathbb{R}\times\{-2\rho\}$. 
Now for any $t\in [-\lambda^{-1}, \lambda^{-1}]$, let us define the following two regions (Figure~\ref{fig:lower-bound-reach-1}):

\begin{itemize}
  \item % $\bullet$
    $R_1(t) = Q_1\setminus [\bfB((0, \lambda^{-1}), 1/\lambda)\cup \bfB((0, -\lambda^{-1}), \lambda^{-1})]$
    where $Q_1$ is the closed rectangle defined by the four vertices $(-t, \lambda^{-1}), (-t, -\lambda^{-1}), (t, \lambda^{-1})$, and $(-t, -\lambda^{-1})$, and

  \item % $\bullet$ 
    $R_2(t) = Q_2\setminus [B_{\lambda^{-1}}((0, \lambda^{-1} - \rho))\cup B_{\lambda^{-1}}((0, -\lambda^{-1} - \rho))]$ where $Q_2$ is the closed rectangle defined by the four vertices $(-t, \lambda^{-1} - \rho), (-t, -\lambda^{-1} - \rho), (t, \lambda^{-1} - \rho)$, and $(-t, -\lambda^{-1} - \rho)$.
\end{itemize}

\begin{figure}[ht!]
\centering
\includegraphics[width=10cm]{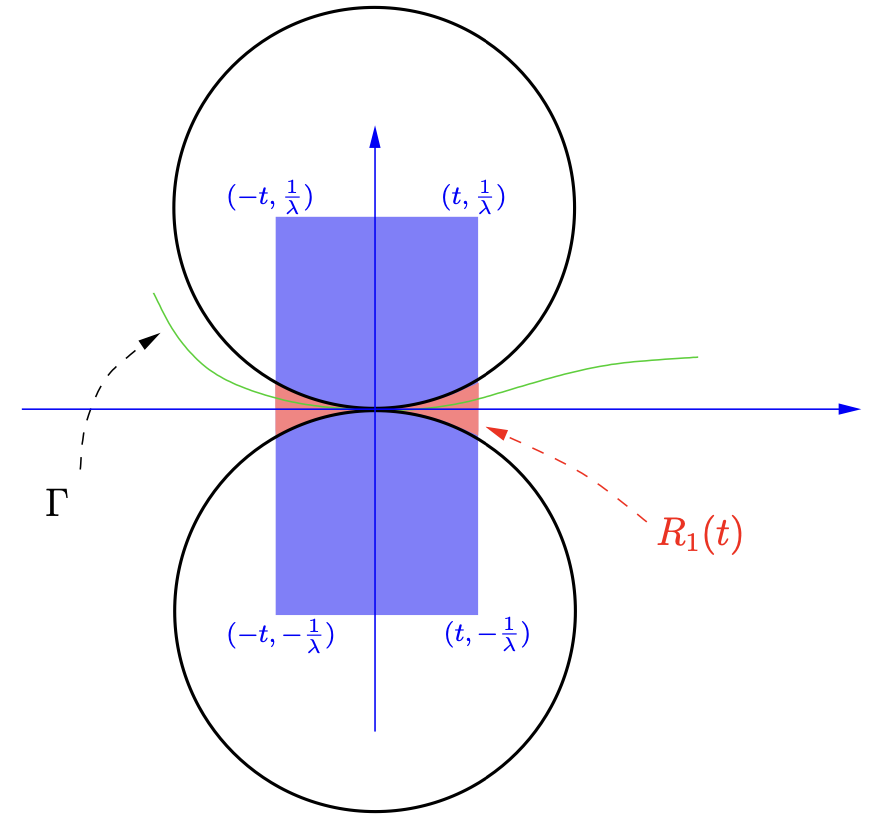}
\caption{The region $R_1(t)$ is constructed by taking the interior of the two disks away from the rectangle.}
\label{fig:lower-bound-reach-1}
\end{figure}

The curvature of $\Gamma$ is defined almost everywhere in $\mathbb{S}^1$ and is bounded above by $\lambda$~\cite{allard-2007-1}. Thus, for any $s \in \mathbb{S}^1$ and any $x \in (s - \pi\lambda^{-1}, s + \pi\lambda^{-1})$, $\gamma(x)$ is contained outside the two unique balls of radius $\lambda^{-1}$  which have first order contact with $\Tan(\Gamma, q)$.

Therefore for $\delta \in (0, \frac{\pi}{2\lambda})$, there exists functions $F$ and $F'$ defined on the first coordinate axis for which the images of $[v - \delta, v + \delta]$ and $[v' - \delta, v' + \delta]$ under $\gamma$ equal the graph of $F$ and $F'$ respectively. Hence, for any point $t \in [-\lambda^{-1}, \lambda^{-1}]$, we get that $(t,F(t))$ and $(t,F'(t))$ are contained inside the regions $R_1(t)$ and $R_2(t)$ respectively (see again Figure~\ref{fig:lower-bound-reach-1}).

%\begin{figure}[htp!]
%  \centering
%\scalebox{1}{\input{../Figures/R-region.pdf_t}}  
%  \caption{The region $R_1(t)$ is constructed by taking the interior of the two disks away from the rectangle.}
%  \label{fig:neck2}
%\end{figure}

For any $t \in [-\lambda^{-1}, \lambda^{-1}]$, let
\begin{align*}
  \Gamma^1(t) &:= \{(x, F(x)) : x\in [-t\lambda^{-1}, t\lambda^{-1}]\}
\end{align*}

and
\begin{align*}
  \Gamma^2(t) &:= \{(x, F'(x)) : x\in [-t\lambda^{-1}, t\lambda^{-1}]\}
\end{align*}

be the local images of $\Gamma$ contained in $R_1$ and $R_2$, respectively. For $0 < s \leq t$, let $L(s)$ denote the portion of the vertical line segment $x = s$ that lies between $\Gamma^1(t)$ and $\Gamma^2(t)$, and let $S^\ast(t)$ be the bounded region whose boundary is the union $\Gamma^1(t)\cup \Gamma^2(t)\cup L(-t)\cup L(t)$ with orientation induced by $\Gamma$.

\vspace{2mm}

\item % {\bf Step 3.} 
  In this step, we will find the lower bound for $\reach(\Gamma)$.

  %\begin{figure}[htp!]
  %  \centering
  %\scalebox{.9}{\input{../Figures/comparison.pdf_t}}  
  %  \caption{The area of the region $S^{\ast}(t)$ will be bounded above by the area of the butterfly wings region.}
  %  \label{fig:neck3}
  %\end{figure}

  For all $x \in (0, \lambda^{-1})$, we have the lower bound
  \begin{align}\label{eq:step3-eq1}
    \mathbf{M}(\Gamma^1(x)+ \Gamma^2(x)) &\geq 4x
  \end{align} 
 
and the upper bound
\begin{align}\label{eq:step3-eq2}
(4y +2(2\rho)) + \lambda(2(2\rho x + xy)) &\geq \mathbf{M}(L(x) + L(-x)) + \lambda\mathbf{M}(S^\ast(x)). 
 \end{align}
 
Therefore, since $\mathbf{M}(\Gamma) = \mathbf{M}(\Gamma^1(x) + \Gamma^2(x)) + \mathbf{M}(\Gamma - (\Gamma^1(x) + \Gamma^2(x)))$, and $\mathbf{M}(\Gamma - \partial S^\ast) = \mathbf{M}(\Gamma - (\Gamma^1(x) + \Gamma^2(x))) + \mathbf{M}(L(x) + L(-x))$, showing that $\mathbf{M}(\Gamma) > \mathbf{M}(\Gamma-\partial S^\ast(x)) + \lambda\mathbf{M}(S^\ast(x))$ reduces to showing 
\begin{align}\label{eq:step3-eq3}
\mathbf{M}(\Gamma^1(x) + \Gamma^2(x)) &> \mathbf{M}(L(x) + L(-x)) + \lambda\mathbf{M}(S^\ast(x)).
\end{align}

Since conditions in Equations (\ref{eq:step3-eq1}) and (\ref{eq:step3-eq2}) are true for all $x \in (0, \lambda^{-1})$, we will show inequality in Equation (\ref{eq:step3-eq3}) by finding values of $x$ such that for $y(x) = \sqrt{(1/\lambda)^2 - x^2} - 1/\lambda$, we have

\begin{align}
4x & > (4y +2(2\rho)) + \lambda(2(2\rho x + xy));
\end{align}
 or equivalently,

\begin{align}\label{eq:step3-eq4}
\displaystyle\frac{4x - 2y(\lambda x + 2)}{4(\lambda x + 1)} &> \rho.
\end{align}

By changing to polar coordinates, the condition in Equation (\ref{eq:step3-eq4}) is equivalent to

\begin{align}\label{eq:step3-eq5}
\displaystyle\left(\frac{1}{\lambda}\right)\frac{2\cos{\theta} - (1 + \sin{\theta})(\cos{\theta} + 2)}{2(\cos{\theta} + 1)} &> \rho.
\end{align}

Indeed, we may choose any $\theta \in (3\pi/2, 2\pi)$ for which the inequality~(\ref{eq:step3-eq5}) holds true.
However, since we are trying to obtain the largest lower bound possible, we find that for 

\begin{align}
C(\theta) &= \frac{2\cos{\theta} - (1 + \sin{\theta})(\cos{\theta} + 2)}{2(\cos{\theta} + 1)},
\end{align}
we get that
\begin{align*}
\hat{C}&:= \sup_{\theta \in (3\pi/2, 2\pi)}C(\theta)\\
&\approx 0.2217
\end{align*}
and that 
\begin{align*}
\theta' &:= \mathrm{argmin}\{C(\theta) : \theta \in (3\pi/2, 2\pi)\}\\
&\approx 5.231. 
\end{align*}

Therefore, letting $S^\ast = S^\ast(x')$ for $x' = \lambda^{-1}\cos(\theta')$ concludes the proof.
\end{enumerate}
\end{proof}

\begin{theorem}\label{thm: flatNormReach}
  Suppose that $T = \partial\Omega$ for $\Omega\subset\bbR^2$ and that
  $\Gamma = T-\partial S$ is a corresponding one dimensional
  multiscale-flatnorm minimizer with scale parameter $\lambda$. Then
  the reach of $\Gamma$ is bounded below by $\hat{C}/\lambda$.
\end{theorem}

\begin{proof}
  Since $\Gamma$ is a $C^{1,1}$ embedded curve in $\bbR^2$ with empty
  boundary, it has positive reach. This implies that there is some
  small $\epsilon$ so that every component of $\Gamma$ contains a
  translate of $\bfB(0,\epsilon)$. Since the length of $\gamma$ is
  finite, we must also have that $\Gamma\subset \bfB(0,N)$ for some
  large enough $N\in\bbN$. Since each of the balls in separate
  components of $\Gamma$ are disjoint, then there are most $(\pi
  N^2)/(\pi \epsilon^2)$ components of $\Gamma$. Since each of the
  components of $\Gamma$ are embedded circles, each with the same
  orientation, we get that the comparison construction in
  Theorem~\ref{thm:one-circle} goes through whether or not the points
  $\gamma_i(v)$ and $\gamma_j(v')$ are on the same circle or not.
  Therefore, we can use the argument used to prove
  Theorem~\ref{thm:one-circle} to finish the proof.
\end{proof}

We will now use the theorem from above in the proof of the following theorem, which states that we may use the volume of flat norm minimizers in a lower bound of the probability that our balls \emph{almost} covers $E$.

\begin{theorem}\label{prop:flatnormsufficient}
  Let $E \subset \bbR^2$ be an open set of finite perimeter. 
  If $\lambda > \Lambda_E$ and $\delta > 0$ are chosen such that 
  \[|S_\lambda| < \frac{\delta^2}{2}\quad \text{ and } \quad 0 < \delta < \frac{1}{5\lambda} \quad \text{ for } \quad S_\lambda \in \srF_\lambda(\partial E),\]
  then for $A := E\cap (E- S_\lambda)$ and $\alpha_\delta := \delta^2/2|E|$ there exists a good $\alpha_\delta$-almost partition $\clR_{A}$ of $E$ such that
  \begin{align}\label{eq:almostall}
    P\left(\frac{|\bfB(\bbX, 3\delta) \cap E|}{|E|} \geq 1 - \alpha_\delta\right) \geq 1 - M\exp\left(\frac{-(\delta^2 2^{-1} - |S_\lambda|}{|A|} N\right)
  \end{align}
  where $N$ is the number of samples in $\bbX$ and $M$ is the cardinality of $\clR_{A}$.
\end{theorem}

\begin{proof}
  Let $E$ be an open set of finite perimeter in $\bbR^2$.
  Suppose that there exists a $\delta > 0$ and $\lambda > \Lambda_E$ such that 
  \[|S_\lambda| < \frac{\delta^2}{2}\quad \text{ and } \quad 0 < \delta < \frac{1}{5\lambda}\,.\]
  By Theorem~\ref{thm: flatNormReach}, we get that $\reach(\partial E_\lambda) > 1/(5\lambda)$ for $E_\lambda := E - S_\lambda$. 
  Since $\delta < \reach(\partial E_\lambda) \leq \reach(E_\lambda^c)$, by Lemma~\ref{lem:opening}, there exists a partition $\clR_{E_\lambda}^\delta$ of $E_\lambda$ such that 
  \[|R| \geq \frac{\delta^2}{2}\quad \text{ and }\quad \diam(R) \leq 3\delta\quad \text{ for all } R \in \clR_{E_\lambda}^\delta.\]

  \noindent Now take the partition of $E\cap (E - S_\lambda)$ defined by 
  \[\clR_{E\cap (E - S_\lambda)} = \{R \cap (E\cap (E - S_\lambda)) : R \in \clR^\delta_{E_\lambda}\}\,.\]
  Since $|E\cap (E - S_\lambda)| \geq |E| - |E - S_\lambda| \geq |E| - \frac{\delta^2}{2}$ implies that
  \[\frac{|E \cap(E - S_\lambda)|}{|E|} \geq 1 - \frac{\delta^2}{2|E|}\,,\]
  if we can prove that 
  \begin{align}\label{eq:want}
    P(\bfB(\bbX, 3\delta) \supset E\cap E_\lambda) \geq 1 - M\exp\left(\frac{-(\delta^2 2^{-1} - |S_\lambda|)}{|E\cap E_\lambda|} N\right)
  \end{align}
  we will have proven the inequality in Equation (\ref{eq:almostall}). 
  Since $|S_\lambda| < \delta^2/2$, we get that $\inf\{|R| : R \in \clR_{E \cap E_\lambda}\} \geq \inf\{|R| : R \in \clR^\delta_{E_\lambda}\} - |S_\lambda| > 0$.
  Therefore, we may apply the argument in Theorem~\ref{thm:reach} and obtain the inequality in Equation (\ref{eq:want}).
\end{proof}

In the special case that $E \subset \bbR^2$ is obtained in a similar manner to Corollary~\ref{cor:UminusA}, that is, $E = U \setminus A$ where $A$ is a subset of a compactly supported subset of $U$, then for a range of scales $\lambda$ there exists $\eta > 0$ such that $S_\lambda = A$ if $|A| < \eta$. 
In particular, running the multiscale flat norm on $\partial E$ will fill $U\setminus A$ back in to get $U$. 
Notice that if $|A| < \delta^2/2$, then Proposition~\ref{prop:flatnormsufficient} tells us that we have a good \emph{almost} partition of $E$. 
However, in Proposition~\ref{thm:fnreach} we prove that it would in fact be a \emph{good} partition of $E$ since $S_\lambda = A$. 

For subsets $A$ and $B$ of $\bbR^\dmn$ we will use $A \subset\subset B$ to denote that $A$ is a compactly supported subset of $B$. 

\begin{figure}[htp!]
  \centering 
  \includegraphics[width=.6\linewidth]{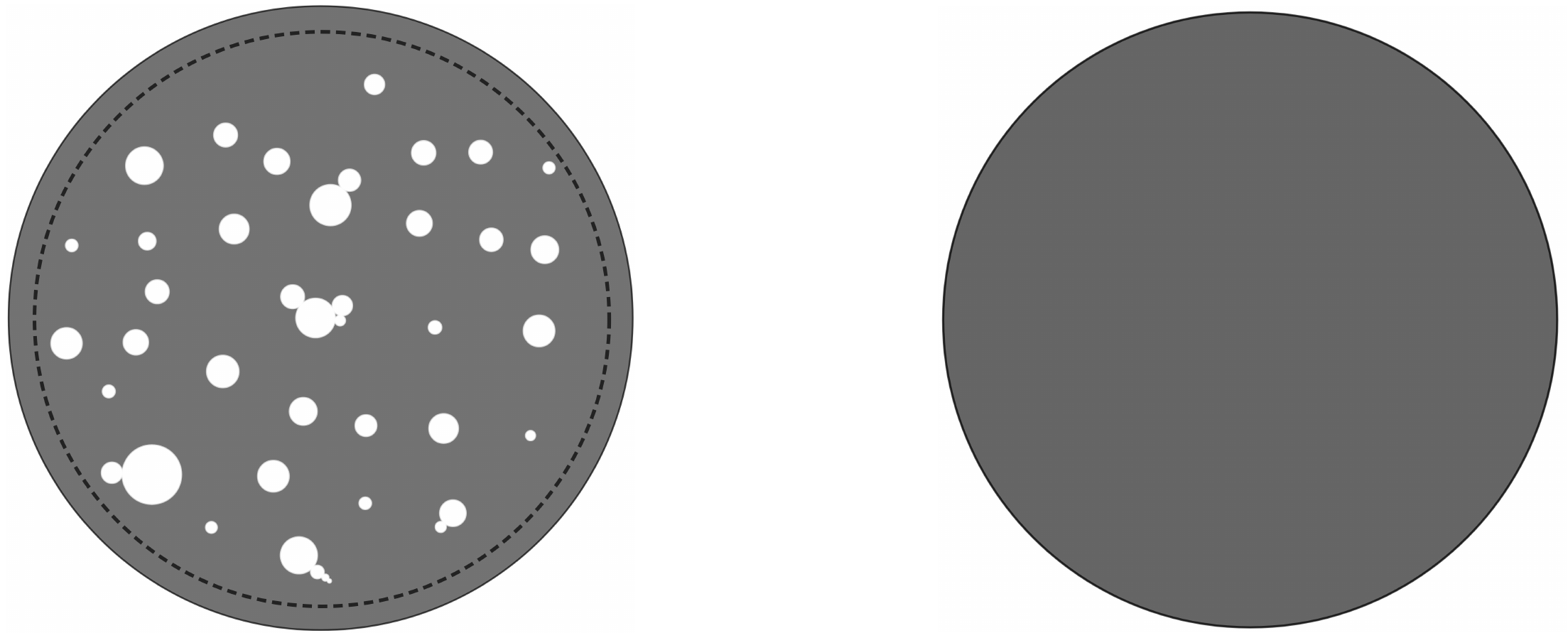}
  \caption{On the left we have removed a small set $A$ from $U = \bfB(0, R)$ away from the boundary of $U$ to obtain $E = U\setminus A$ such that $\reach(E^c) = 0$. 
  Although $E$ does not satisfy the assumptions of Lemma~\ref{lem:opening}, we may apply the multiscale flat norm to $\partial E$ to obtain the nice set $U$ with positive reach on the right.
  \label{fig:flatnormfill}
  }
\end{figure}

The proof of Theorem~\ref{thm:fnreach} uses Lemmas~\ref{lem: vixie-1},~\ref{lem: lemma_vixie}, Corollary~\ref{cor: corollary}, Proposition~\ref{prop: flat norm}, and Theorem~\ref{thm:vixie} from our previous work~\cite{Vi2007}.
We reproduce these results here for the sake of completeness.
Recall the notions of measure theoretic boundary, interior, and exterior given in Definition \ref{def: boundary}.
\begin{lemma}\label{lem: vixie-1}
Let $A$ be a subset of $\bbR^n$ with finite perimeter. 
Then 
\begin{enumerate}
\item $\Per(A) = \clH^{n - 1}(\partial_\ast A)$, and 
\item $\bbR^n = A_\ast^o \cup \partial_\ast A \cup A_\ast^i$ and the three sets are pairwise disjoint. 
\end{enumerate}
\end{lemma}
\begin{proof}
The first conclusion follows from the results of Evans and Geriepy~\cite[\S 5.7 Theorem 2, \S 5.8 Lemma 1]{evans-1992-1}. 
The second conclusion follows directly from Definition~\ref{def: boundary}.
\end{proof}

\begin{lemma}\label{lem: lemma_vixie}
The following statements are true for subsets $A$ and $B$ of $\bbR^n$:
\begin{enumerate}
\item If $x \in A_\ast^i$ or $x \in B_\ast^i$ then $x \in (A \cup B)_\ast^i$. 
\item $\partial_\ast(A \cup B) \subset \partial_\ast A \cup \partial_\ast B$. 
\item $\partial_\ast A \cup \partial_\ast B = (\partial_\ast A \cap B_\ast^i) \cup (\partial_\ast A \cap B_\ast^o) \cup (\partial_\ast B \cap A_\ast^i) \cup (\partial_\ast B \cap A_\ast^0) \cup (\partial_\ast A \cap \partial_\ast B)$.
\item (1-3) above immediately imply that $\partial_\ast(A \cup B) \subset (\partial_\ast A \cap B_\ast^o) \cup (\partial_\ast \cap A_\ast^o)\cup (\partial_\ast A \cap \partial_\ast B)$.
\item $(\partial_\ast A\cap B_\ast^o) \cup (\partial_\ast B \cap A_\ast^o) \subset \partial_\ast(A \cup B)$.
\item $\partial_\ast A^c = \partial_\ast A$.
\item $(A^c)_\ast^0 = A_\ast^i$.
\end{enumerate}
\end{lemma}
\begin{proof}
The lemma follows in a straightforward manner from the definitions of measure theoretic boundary, interior, and exterior in Definition \ref{def: boundary}.
\end{proof}

\begin{corollary}\label{cor: corollary}
If $\clH^{n-1}(\partial_\ast A \cap \partial_\ast B) = 0$ then 
\begin{enumerate}
\item $\clH^{n-1}(\partial_\ast(A \cup B)) = \clH^{n-1}(\partial_\ast A \cap B_\ast^o) + \clH^{n-1}(\partial_\ast B \cap A_\ast^o)$ and
\item $\clH^{n-1}(\partial_\ast(A \cap B)) = \clH^{n-1}(\partial_\ast A \cap B_\ast^i) + \clH^{n-1}(\partial_\ast B \cap A_\ast^i)$.
\end{enumerate}
\end{corollary}
\begin{proof}
{\bf 1.} Lemma~\ref{lem: lemma_vixie} (4--5) imply that 
\begin{align}
\clH^{n-1}(\partial_\ast A \cap B_\ast^o) &+ \clH^{n - 1}(\partial_\ast B \cap A_\ast^o)\\
&\leq \clH^{n-1}(\partial_\ast(A \cup B))\\
&\leq \clH^{n-1}(\partial_\ast A \cap B_\ast^o) + \clH^{n-1}(\partial_\ast B \cap A_\ast^o) + \clH^{n-1}(\partial_\ast A \cap \partial_\ast B)
\end{align}
and the conclusion follows. 

{\bf 2.} This follows from the previous conclusion {\bf 1.}, Lemma~\ref{lem: lemma_vixie} (6--7), and the fact that $A \cap B = (A^c \cup B^c)^c$. 
\end{proof}

\begin{remark}
Since $\partial_\ast A = (\partial_\ast A \cap B_\ast^i) \cup (\partial_\ast A \cap \partial_\ast B) \cup (\partial_\ast A \cap B_\ast^o)$, the assumption that $\clH^{n-1}(\partial_\ast A \cap \partial_\ast B) = 0$ means we can, for the same measurement, consider $\partial_\ast A = (\partial_\ast A \cap B_\ast^i)\cup (\partial_\ast A\cap B_\ast^o)$. 
\end{remark}

\begin{remark}
Now suppose that $\clH^{n-1}(\partial_\ast A) < \infty$ and $\bfU_r$ is the ball of radius $r$ centered at $x \in \bbR^n$ (we suppress the $x$). 
Note that $\partial_\ast \bfU_r = \partial \bfU_r$. 
By the coarea formula, the set of $r$'s such that $\clH^{n-1}(\partial_\ast \bfU_r \cap \partial_\ast A) > 0$ is at most countable. 
We conclude that the $r$'s for which $\clH^{n-1}(\partial_\ast \bfU_r \cap \partial_\ast A) = 0$ are dense and have full measure in $\bbR$. 
For the rest of this section we assume that we have chosen $r$ such that $\clH^{n-1}(\partial_\ast A \cap \partial_\ast \bfU_r) = 0$. 
\end{remark}

\begin{proposition}\label{prop: flat norm}
Suppose $\bfU_r \subset E \subset \bbR^n$. 
Then 
\[
\Delta E := E(\Sigma \cup \bfU_r) - E(\Sigma) = -\clH^{n-1}(\partial_\ast \Sigma \cap (\bfU_r)_\ast^i) + \clH^{n-1}(\partial_\ast \bfU_r \cap \Sigma_\ast^o) - \lambda|\bfU_r\setminus \Sigma|. 
\]
\end{proposition}
\begin{proof}
Since $\Per(\Sigma\cup \bfU_r) = \clH^{n-1}(\partial_\ast \Sigma \cap (\bfU_r)_\ast^o) + \clH^{n-1}(\partial_\ast \bfU_r \cap \Sigma_\ast^o)$ and $\Per(\Sigma) = \clH^{n-1}(\partial_\ast \Sigma \cap (\bfU_r)_\ast^i) + \clH^{n-1}(\partial_\ast\Sigma \cap (\bfU_r)_\ast^o)$ we get $\Per(\Sigma \cup (\bfU_r)_\ast^i) - \Per(\Sigma) = -\clH^{n-1}(\partial_\ast \Sigma \cap (\bfU_r)_\ast^i) + \clH^{n-1}(\partial_\ast \bfU_r \cap \Sigma_\ast^o)$. 
Noting that $\bfU_r \subset \Omega$ implies $|(\Sigma \cup \bfU_r)\Delta \Omega| - |\Sigma \Delta \Omega| = |\bfU_r \setminus \Sigma|$ finishes the proof. 
\end{proof}

The next theorem will be used in the proof of Theorem~\ref{thm:fnreach}. 
It says the following. 
Whenever $\Sigma_\lambda$ is a flat norm minimizer with scale $\lambda > 0$ of $\partial(\bfE^n\myell E)$, and a ball of radius $2/\lambda + \epsilon$ is contained in $E$, then the concentric ball of radius $2/\lambda - \epsilon$ is completely contained in the flat norm minimizer $\Sigma_\lambda$. 

\begin{theorem}\label{thm:vixie}
Let $R := 2/\lambda$. 
Given $\hat{r} \in (R, \frac{\sqrt{7}}{2}R)$ and $\epsilon \in (0, 1 - \frac{1}{\sqrt{2}})$, we can choose $\delta = \delta(R, \hat{r}, \epsilon) > 0$ such that 
\[
|B_{\hat{r}}\setminus \Omega| < \delta \Rightarrow B_{(1 - \epsilon) R} \subset \Sigma.
\]
\end{theorem}
\begin{proof}
In the case that $B_r \cap \Omega^c \neq \emptyset$, 
\begin{align}
\nabla E &= -\clH^1(\partial_\ast \Sigma \cap B_{r^\ast}^i + \clH^1(\partial_\ast B_r \cap \Sigma_\ast^o)) - \lambda|B_r \setminus \Sigma| + 2\lambda|B_r \cap \Omega^c \cap \Sigma^c|\\
&\leq -\clH^1(\partial_\ast \Sigma \cap B_{r_\ast}^i) + \clH^1(\partial_\ast B_r \cap \Sigma_\ast^o) - \lambda|B_r \setminus \Sigma | + 2\lambda|B_r \cap \Omega^c|\\
&= -\clH^1(\partial_\ast \Sigma \cap B_{r_\ast}^i) + \clH^1(\partial_\ast B_r \cap \Sigma_\ast^o) - \lambda|B_r \setminus \Sigma| + 2\lambda|B_r \setminus \Omega|\\
&= (\clH^1(\partial_\ast B_r) - \lambda|B_r|) + (\lambda|B_r \cap \Sigma| - \clH^1(\partial_\ast(B_r \cap \Sigma))) + 2\lambda|B_r\setminus\Omega|\\
&= (\mathrm{Per}(B_r) - \lambda|B_r|) + (\lambda|B_r \cap \Sigma| - \mathrm{Per}(B_r \cap \Sigma)) + 2\lambda|B_r \setminus \Omega|.\label{eq:22}
\end{align}

{\bf Claim 1:} $|B_r \setminus \Sigma| \leq 6\delta$. 

{\it Proof of Claim 1:} Since we assume $\Sigma$ is a minimizer, $\nabla E \geq 0$. 
We will perturb with balls of radius $r \leq \wedge{r}$. 
Then, $|B_r\setminus \Omega| \leq |B_{\wedge{r}} \setminus \Omega | := \delta$. 
These assumptions together with (\ref{eq:22}) and the isoperimetric inequality $(\mathrm{Per}(B_r\cap \Sigma) \geq 2\sqrt{\pi} |B_r \cap \Sigma|^{1/2})$ imply: 
\begin{align*}
0 &\leq \nabla E \leq 2\pi r - \lambda \pi r^2 + \lambda |B_r \cap \Sigma| - 2 \sqrt{\pi} |B_r \cap \Sigma|^{1/2} + 2\lambda\delta\\
&= \lambda |B_r \cap \Sigma| - 2 \sqrt{\pi} |B_r \cap \Sigma|^{1/2} + \left(2\lambda\delta +  2\pi r - \lambda \pi r^2\right)\\
&= f(\xi) := \lambda \xi^2 - 2\sqrt{\pi}\xi + \left(2\lambda\delta + 2 \pi r - \lambda \pi r^2\right), (\xi := |B_r \cap \Sigma|^{1/2})\\
&= \frac{2}{R} \xi^2 - 2 \sqrt{\pi}\xi + \left(\frac{4\delta}{R} + 2\pi r - \frac{2\pi r^2}{R}\right) (\mathrm{recalling\ } R = 2/\lambda).
\end{align*}
In view of the last inequality, we describe values of $\xi$ for which $f(\xi) \geq 0$. 
For a given $r$, the zeros of $f(\xi)$ are at: 
\begin{equation}\label{eq:2}
\xi_{\pm}(r) = \sqrt{\frac{\pi R^2}{4}} \pm \sqrt{\frac{\pi R^2}{4} + (\pi r(r - R)) - 2\delta}.
\end{equation}
Thus, for all $r \leq \wedge{r}$, we have that either 
\[
|B_r\cap \Sigma|^{\frac{1}{2}} \leq \xi_-(r), \quad \text{ or } \quad |B_r \cap \Sigma|^{\frac{1}{2}} \geq \xi_+(r).
\]
If we take $r = \wedge{r} > R$, then $2\pi r(r - R) > 0$, and assuming 
{\bf Condition 1.} 
\[
\delta < \frac{\pi \wedge{r}}{2}(\wedge{r} - R)
\]
implies $\xi_-(r) < 0$. 
This implies 
\[
|B_{\wedge{r}} \cap \Sigma| \geq \xi_+^2(\wedge{r}) > \pi R^2.
\]
Since $r = \wedge{r} < \frac{\sqrt{7}}{2} R$ we get 
\begin{equation}\label{eq:31}
|B_R \cap \Sigma | > |B_r \cap \Sigma| - \frac{3\pi R^2}{4} > \frac{\pi R^2}{4}.
\end{equation}
Now we consider $\xi_{\pm}(R)$: 
\begin{equation}\label{eq:32}
\xi_{\pm}(R) = \sqrt{\pi R^2}{4} \pm \sqrt{\frac{\pi R^2}{4} - 2\delta}.
\end{equation}
Assuming 
{\bf Condition 2.} 
\[
\delta < \frac{\pi R^2}{8},
\]
we get that $\xi_{\pm}(R)$ are real and distinct. 
Since 
\[
\xi_-^2(R) < \frac{\pi R^2}{4} < |B_r \cap \Sigma|,
\]
we conclude that 
\[
\xi_+^2 \leq |B_r \cap \Sigma|.
\]
Computing we get 
\begin{align}
\xi_+^2 &= \left(\sqrt{\frac{\pi R^2}{4}} + \sqrt{\frac{\pi R^2}{4} - 2\delta}\right)^2\label{eq:36}\\
&= \frac{\pi R^2}{2} - 2\delta + \frac{\pi R^2}{2}\sqrt{1 - \frac{8\delta}{\pi R^2}}\label{eq:37}\\
&\geq \frac{\pi R^2}{2} - 2\delta + \frac{\pi R^2}{2}\left(1 - \alpha \frac{8\delta}{\pi R^2}\right) \left(\mathrm{assuming\ } \delta \leq \frac{\pi R^2}{8} \frac{2\alpha - 1}{\alpha^2}\right)\label{assumption 38}\\
&= \pi R^2 - (2 + 4\alpha)\delta.\label{eq:39}
\end{align}
Choosing $\alpha = 1$, and noting that Condition $2$ then implies the assumption in~(\ref{assumption 38}) is satisfied, we get 
\[
|B_R \cap \Sigma| \geq \xi_+^2 \geq \pi R^2 - 6\delta.
\]
This gives 
\[
|B_r \setminus \Sigma| \leq 6\delta
\]
as advertised.
This concludes the proof of \emph{Claim 1.}

\begin{remark}
What if either $\hat{r}$ or $R$ are radii such that (\ref{eq:22}) (and therefore (\ref{eq:2})) does not hold?
We can simply choose another $\tilde{r} < \hat{r}$ arbitrarily close to $\hat{r}$, for which (\ref{eq:22}) does hold. 
The $\tilde{\delta} := |B_r\setminus \Omega|$ will be no greater than, and arbitrarily close to, $\delta$. 
As we will see, the only conditions on $\delta$ that are not functions of $R$ and $\epsilon$ are those in Condition 1. 
Therefore, if we replace Condition 1 with 
\[
\delta < \frac{\pi R}{4} (\hat{r} - R)
\]
we know that the delta chosen for any $\hat{r}$ will permit us to arrive at the conclusions of this lemma, even in cases where we have to perturb $\hat{r}$. 
Next we choose a sequence of $r_i > R$ converging monotonically to $R$ for which the inequality does work. 
Inequalities in (\ref{eq:31}) is still valid if we replace $B_R$ with $B_{r_i}$. 
Equation (\ref{eq:32}) can be slightly modified using (\ref{eq:2}) to 
\begin{equation}\label{eq:43}
\xi_{\pm}(r_i) = \sqrt{\frac{\pi R^2}{4}} \pm \sqrt{\frac{\pi R^2}{4} - 2\delta_i}
\end{equation}
where the $\delta_i < \delta$ and $\delta_i \to \delta$ as $i \to \infty$. 
Now, simply repeating the derivation in lines (\ref{eq:36}) to (\ref{eq:39}), gives 
\begin{equation}\label{eq:44}
|B_R \cap \Sigma| = \lim_{i\to\infty} |B_{r_i}\cap \Sigma| \geq \pi R^2 - 6\delta = \lim_{i\to\infty} \pi R^2 - 6\delta_i.
\end{equation}
\end{remark}
Now we continue with the proof. 
Computing (again and less optimally, but sufficiently for our purposes) the change in energy when we add a ball $B_r$ to $\Sigma$ for $r \in (0, R)$, we get 
\begin{align}
\Delta E = E(\Sigma \cup B_r) &- E(\Sigma) \label{eq:47} \\
&\leq -\clH^1(\partial_\ast \Sigma \cap B_{r_\ast}^i) + \clH^1(\partial_\ast B_r \cap \Sigma_\ast^o) + \lambda|B_r \setminus \Sigma|\\
&= -\mathrm{Per}(\Sigma; B_r) + \clH^1(\partial_\ast B_r \cap \Sigma_\ast^o) + \lambda|B_r\setminus \Sigma|.
\end{align}
By the coarea formula and properties of the measure theoretic exterior, 
\begin{equation}\label{eq:48}
|B_r \setminus \Sigma| = |B_r \cap \Sigma_\ast^o| = \int_0^r \clH^1(\partial_\ast B_\xi \cap \Sigma_\ast^o)\,d\xi.
\end{equation}
By the relative isoperimetric inequality applied in the ball $B_r(x_0)$, 
\begin{equation}\label{eq:49}
\mathrm{Per}(\Sigma; B_r) \geq C \min\left\{\|B_r \setminus \Sigma|^{1/2}, |\Sigma \cap B_r|^{1/2}\right\}.
\end{equation}
Assuming 

{\bf Condition 3.} $6\delta < \frac{1}{4} \pi R^2$

implies $|B_R\setminus \Sigma| < \frac{1}{4} \pi R^2$. 
Assuming $r > \frac{R}{\sqrt{2}}$ implies that $|B_r\setminus \Sigma| \leq |\Sigma \cap B_r|$ and consequently 
\begin{equation}\label{eq:50}
\mathrm{Per}(\Sigma; B_r) \geq C|B_r\setminus \Sigma|^{1/2}.
\end{equation}
This gives a condition on $\epsilon$:

{\bf Condition 4.} $\epsilon < 1 - \frac{1}{\sqrt{2}}$.

Define $v(r) := |B_r \setminus \Sigma|$. 
By differentiating (\ref{eq:48}) with respect to $r$, and using (\ref{eq:50}) we see that the inequality concerning the change in energy given in (\ref{eq:47}) can be written as 
\begin{equation}
\label{eq:51}
E(\Sigma\cup B_r) - E(\Sigma) \leq \lambda v(r) - C\sqrt{v(r)} + v'(r).
\end{equation}
We will use the differential expression on the right to show that the change in energy on the left has to be negative for some $r$ close to $R$. 
\begin{remark}
Note that by choosing $\delta$ small enough, we can make $v(r)$ arbitrarily small and obtain $\lambda v(r) - C\sqrt{v(r)} < 0$; if the right hand side is positive then we have $v'(r) > 0$. 
This in turn means that $v(r)$ decreases as $r$ gets smaller. 
We exploit this to force the right hand side to zero. 
\end{remark}

{\bf Lemma.} $v'(r) - C\sqrt{v(r)} + \lambda v(r) \leq 0$ for a set of $r \in ((1 - \epsilon)R, R)$ with positive measure. 
{\it proof of lemma.} Assume 
\begin{equation}\label{eq:52}
v'(r) - C\sqrt{v(r)} + \lambda v(r) \geq 0 \text{ for a.e. } r \in ((1 - \epsilon)R, R),
\end{equation}
otherwise we are done. 
Let $w(s) := e^{-\lambda s} v(R - s).$
Then (\ref{eq:52}) turns into 
\begin{equation}
\label{eq:53}
w'(s) + Ce^{\frac{-\lambda s}{2}}\sqrt{w(s)} \leq 0 \text{ for a.e. } s \in (0, \epsilon R)
\end{equation}
with the initial condition $w(0) = |B_r\setminus \Sigma|$ and $w(s) \geq 0$. 
Solutions of this differential inequality can be bounded from above by solutions of the following differential equality:
\begin{equation}\label{eq:54}
\bar{w}' = -Ce^{-\lambda s}{2}\sqrt{\bar{w}}, \quad 
\bar{w}(0) = |B_R\setminus \Sigma|,\quad  \bar{w} \geq 0.
\end{equation}
The solution is 
\begin{align*}
\sqrt{\bar{w}(s)} &= \max\left(0, \frac{C}{\lambda}\left(e^{-\frac{\lambda}{2}s} - 1\right) + \sqrt{|B_R\setminus \Sigma|}\right)\\
&= \max\left(0, \frac{CR}{2}\left(e^{-\frac{s}{R}} - 1\right) + \sqrt{|B_R\setminus \Sigma|}\right).
\end{align*}
Therefore if $|B_R\setminus \Sigma| \leq 6\delta$ and 

{\bf Condition 5.} {\it $6\delta \leq \alpha$, where $\alpha$ is any solution to 
\begin{equation}
\frac{CR}{2}\left(e^{-\frac{\epsilon R}{R}} - 1\right) + \sqrt{\alpha} = \frac{CR}{2}\left(e^{-\epsilon} - 1\right) + \sqrt{\alpha} < 0
\end{equation}
i.e., we have $\delta < \frac{C^2R^2}{24}(1 - e^{-\epsilon})^2$,
} 
then we have a set of $r$ with positive measure in $((1 - \epsilon) R, R)$ such that $v(r) = 0$ and $v'(r) = 0$. 
Thus concluding the lemma. 
This lemma immediately implies that for some $r \in ((1 - \epsilon)R, R)$, $B_r \cup \Sigma$ is also a minimizer.
\end{proof}

\begin{theorem}\label{thm:fnreach}
  Let $E = \bfE^2\myell (U \setminus A)$ where $U$ is a bounded open subset of $\bbR^2$ with $\mathrm{reach}(\partial U) > \rho > 0$, and let $W$ be an open, compactly supported subset of $U$.
  If 
  \begin{align}\label{cond:1}
    0 < 2/\lambda < \rho,
  \end{align}
  then 
  \begin{align}\label{conc:thm}
    \exists \delta > 0\ \text{ such that if } A \subset\subset W \text{ with } |A| < \delta,\ \text{ then } \partial (\bfE^2 \myell U) \in \mathrm{arg\,}\srF_\lambda(\partial E).
  \end{align}
\end{theorem}

\begin{proof}
  Let $R = 2/\lambda$. 
  By a result of Allard~\cite{allard-2006-1}, we know that the support of a minimizer $T \in \mathrm{arg\,} \srF_\lambda(\partial E)$ differs from $\partial E$ only at arcs of circles of radius $1/\lambda$, which subtend angles no more than $\pi$, and meet $\partial E$ tangentially. 
  Thus, if we first show that 
  \begin{align}
    A \subset \Sigma_\lambda,
  \end{align}
  the only arc of a circle of radius $1/\lambda$ in $T$ that differs from $\partial E$ must start and end at two points of $\partial U$. 
  If this circle touches $T$ at two points, it contradicts the fact that $\reach(\partial U) > 1/\lambda$. 
  Therefore, $T$ must not differ from $\partial U$, i.e., $\partial(\bfE^2\myell U) \in \mathrm{arg\,}\srF_\lambda(\partial E)$.  

  We will prove that $A \subset \Sigma_\lambda$ by showing that for an $\hat{r} > 0$ (to be chosen later) there exists an $\epsilon > 0$ such that for all $a \in A$, we can find a center $x_a \in U$ such that
  \begin{align}\label{thm:fn:pf}
    \bfU(x_a, \hat{r}) \subset U\quad \text{ and }\quad a\in \bfU(x_a, (1-\epsilon)R).
  \end{align}
  By Theorem 2 of the manuscript by Vixie~\cite{Vi2007}, there exists a $\delta = \delta(\hat{r}, \epsilon, R) > 0$, such that 
  \begin{align}
    |\bfU(x_a, \hat{r})\setminus (U\setminus A)| < \delta \quad \text{ implies }\ \quad \bfU(x_a, (1-\epsilon)R) \subset \Sigma_\lambda.
  \end{align}
  Since $\hat{r}$ and $\epsilon$ are independent of $a$, satisfying the first condition of~(\ref{thm:fn:pf}), together with $|A| < \delta$ will be sufficient to apply Vixie's Theorem 2. 
  Thus, every point of $A$ will be contained in a ball that is contained in $\Sigma_\lambda$, and hence we get that $A \subset \Sigma_\lambda$.

  Let us now prove that $A \subset \Sigma_\lambda$.
  For any $R \in (0, \rho)$ we can pick an $\hat{r} > 0$ such that 
  \begin{align}\label{cond:rhat}
    R < \hat{r} < \min\left\{\rho, \,\dist(W, \partial U) + R, \,\frac{\sqrt{7}}{2}R \right\}.
  \end{align}
  We may then pick a small enough $\epsilon > 0$ such that 
  \begin{align}\label{cond:epsilon}
    \hat{r} - (1 - \epsilon)R \,< \,\dist(A, \partial U).
  \end{align}

  Let $a \in A$. 
  Either $\dist(a, \partial U) \geq \rho$, or not. 
  If $\dist(a, \partial U) \geq \rho$, then $\bfU(a, \hat{r}) \subset U$ and clearly $\bfU(a, (1-\epsilon)R)$ contains $a$. 

  If $\dist(a, \partial U) < \rho$, then since $\rho < \mathrm{reach}(\partial U)$ we may let $x_a = \xi(a) + \alpha(a - \xi(a))$ where $\alpha = \hat{r}/|a - \xi(a)|$.
  Since $|\xi(a) - x_a| = \hat{r} < \mathrm{reach}(\partial U)$, and $a - \xi(a) \in \Nor(\partial U, \xi(a))$, part (12) of Theorem 4.8.~in the work of Federer~\cite{federer-1959-1} implies that $\xi(x_a) = \xi(a)$ and hence, since $a \in \bfU(x_a, \hat{r})$, that $\bfU(x_a, \hat{r}) \subset U$. 

  To conclude the proof we must show that $a \in \bfU(x_a, (1 - \epsilon)R)$.
  By the fact that $|x_a - a| + |a - \xi(a)| = \hat{r}$, and that we have chosen $\hat{r}$ and $\epsilon$ to satisfy~(\ref{cond:rhat}) and~(\ref{cond:epsilon}), respectively, we get that 
  \begin{align*}
    |x_a - a| & < \dist(W, \partial U) + R - |a - \xi(a)|\\
              & \leq |a - \xi(a)| + (1 - \epsilon)R - |a - \xi(a)|\\
              & = (1 - \epsilon)R,
  \end{align*}
  thus concluding the proof of the theorem.
\end{proof}

\section{Discussion}

Our work opens up several computational and theoretical questions for further exploration.
Given the reliance of our results on good partitions, can we design an efficient algorithm to compute a finite partition $\{R_i\}$ of given closed bounded set $E$ with nonempty interior where each $R_i$ is a union of dyadic cubes of a certain diameter such that $\clL^n\myell E(R_i) \geq \epsilon$ and $\diam(R_i\cap E) \leq \delta$ for given $\epsilon, \delta > 0$?

The bound $\displaystyle P(B(\bbX, \delta) \supset E) \geq 1 - \sum_{R \in \clR} \exp(-\clL^n(R)N)$ in Equation (\ref{eq:probLB}) motivates the development of methods to find minimizers of
$\displaystyle \sum_{R \in \clR} \exp(-\clL^n(R)N)$
over all partitions $\clR$ of $E$ with $\diam(R) \leq \delta$ and $\clL^n(R) \geq \epsilon$ for all $R \in \clR$.
We also want to use dyadic grids of $\bbR^n$ to obtain our allowed partitions of $E$. 

Our results on $\alpha$-almost partitions of $E$ that allow samples to \emph{almost} cover $E$ with high probability lead to the natural question of computing estimates of $\clL^n(E\setminus E\circ B_r)$ as a function of $r$.
Relating the geometry of $E$ to the size of $\clL^n(E\setminus E \circ B_r)$ would give us conditions under which an $\alpha$-almost cover can be guaranteed to exist.

%\bibliographystyle{plain}
%\bibliography{EA_bibliography,flatnorm,homology}

\begin{thebibliography}{10}

\bibitem{allard-2006-1}
William~K. Allard.
\newblock On the regularity and curvature properties of level sets of
  minimizers for denoising models using total variation regularization; {I}.
  {T}heory.
\newblock {\em \emph{Preprint}}, 2006.

\bibitem{allard-2007-1}
William~K. Allard.
\newblock Total variation regularization for image denoising; {I}. {G}eometric
  {T}heory.
\newblock {\em SIAM Journal on Mathematical Analysis}, 39:1150--1190, 2007.

\bibitem{alvarado2017lower}
Enrique~G. Alvarado and Kevin~R. Vixie.
\newblock A lower bound for the reach of flat norm minimizers.
\newblock {\em arXiv}, 2017.
\newblock \href{http://arxiv.org/abs/1702.08068}{arXiv:1702.08068}.

\bibitem{BoKa2018}
Omer Bobrowski and Matthew Kahle.
\newblock Topology of random geometric complexes: a survey.
\newblock {\em Journal of Applied and Computational Topology}, 1:331--364,
  2018.

\bibitem{BoWe2017}
Omer Bobrowski and Shmuel Weinberger.
\newblock On the vanishing of homology in random {\v{c}}ech complexes.
\newblock {\em Random Structures \& Algorithms}, 51(1):14--51, 2017.

\bibitem{ChStKeMe2013}
Sung~Nok Chiu, Dietrich Stoyan, Wilfrid~S. Kendall, and Joseph Mecke.
\newblock {\em Stochastic Geometry and its Applications}.
\newblock Wiley Series in Probability and Statistics. Wiley, 3rd edition, 2013.
\newblock ISBN: 978-0-470-66481-0.

\bibitem{evans-1992-1}
Lawrence~C. Evans and Ronald~F. Gariepy.
\newblock {\em Measure Theory and Fine Properties of Functions}.
\newblock Studies in Advanced Mathematics. CRC Press, 1992.
\newblock ISBN 0-8493-7157-0.

\bibitem{federer-1959-1}
Herbert Federer.
\newblock Curvature measures.
\newblock {\em Transactions of the American Mathematical Society},
  93(3):418--491, 1959.

\bibitem{federer-1969-1}
Herbert Federer.
\newblock {\em {G}eometric {M}easure {T}heory}.
\newblock Classics in Mathematics. Springer-Verlag, 1969.

\bibitem{federer-1960-1}
Herbert Federer and Wendell~H. Fleming.
\newblock Normal and integral currents.
\newblock {\em Annals of Mathematics}, pages 458--520, 1960.

\bibitem{IbKrVi2013}
Sharif Ibrahim, Bala Krishnamoorthy, and Kevin~R. Vixie.
\newblock Simplicial flat norm with scale.
\newblock {\em Journal of Computational Geometry}, 4(1):133--159, 2013.
\newblock \href{http://arxiv.org/abs/1105.5104}{arXiv:1105.5104}.

\bibitem{Ja1986}
Svante Janson.
\newblock Random coverings in several dimensions.
\newblock {\em Acta Mathematica}, pages 83--118, 1986.

\bibitem{krantz1981distance}
Steven~G Krantz and Harold~R Parks.
\newblock {Distance to $C^k$ hypersurfaces}.
\newblock {\em Journal of Differential Equations}, 40(1):116--120, 1981.

\bibitem{lucas1957submanifolds}
Kenneth~R Lucas.
\newblock {Submanifolds of dimension $n-1$ in $\mathscr{E}^n$ with normals
  satisfying a Lipschitz condition}.
\newblock Technical report, KANSAS UNIV LAWRENCE, 1957.

\bibitem{lytchak2005almost}
Alexander Lytchak.
\newblock Almost convex subsets.
\newblock {\em Geometriae Dedicata}, 115(1):201--218, 2005.

\bibitem{MoVi2007}
Simon~P. Morgan and Kevin~R. Vixie.
\newblock {$L^1$TV} computes the flat norm for boundaries.
\newblock {\em Abstract and Applied Analysis}, 2007:Article ID 45153,14 pages,
  2007.
\newblock \href{http://arxiv.org/abs/math/0612287}{arXiv:0612287}.

\bibitem{niyogi2008finding}
Partha Niyogi, Stephen Smale, and Shmuel Weinberger.
\newblock Finding the homology of submanifolds with high confidence from random
  samples.
\newblock {\em Discrete \& Computational Geometry}, 39(1-3):419--441, 2008.

\bibitem{penrose2021random}
Mathew~D. Penrose.
\newblock Random euclidean coverage from within.
\newblock {\em arXiv}, 2021.
\newblock \href{http://arxiv.org/abs/2101.06306}{arXiv:2101.06306}.

\bibitem{rataj2019curvature}
Jan Rataj and Martina Z{\"a}hle.
\newblock {\em Curvature measures of singular sets}.
\newblock Springer, 2019.

\bibitem{reznikov2016covering}
Aleksandr~B. Reznikov and Edward~B. Saff.
\newblock The covering radius of randomly distributed points on a manifold.
\newblock {\em International Mathematics Research Notices},
  2016(19):6065--6094, 2016.
\newblock \href{http://arxiv.org/abs/1504.03029}{arXiv:1504.03029}.

\bibitem{stein1970singular}
Elias~M. Stein.
\newblock {\em Singular integrals and differentiability properties of
  functions}, volume~2.
\newblock Princeton university press, 1970.

\bibitem{verma2012distance}
Nakul Verma.
\newblock {Distance Preserving Embeddings for General $n$-Dimensional
  Manifolds}.
\newblock {\em Journal of Machine Learning Research}, 14(74):2415--2448, 2013.

\bibitem{Vi2007}
Kevin~R. Vixie.
\newblock Some properties of minimizers for the {C}han-{Esedo\={g}lu $L^1$TV}
  functional.
\newblock {\em arXiv}, 2007.
\newblock \href{http://arxiv.org/abs/0710.3980}{arXiv:0710.3980}.

\bibitem{whitney2015geometric}
Hassler Whitney.
\newblock {\em Geometric integration theory}.
\newblock Princeton university press, 2015.

\end{thebibliography}

\end{document}